\documentclass[12pt,a4paper]{amsart}
\usepackage{pgf,tikz}
\usetikzlibrary{arrows}

\makeatletter
\long\def\unmarkedfootnote#1{{\long\def\@makefntext##1{##1}\footnotetext{#1}}}
\makeatother

\usepackage{amsmath,amssymb}
\usepackage{enumerate,amssymb}
\usepackage{color}
\theoremstyle{plain}
\newtheorem{thm}{Theorem}[section]

\newtoks\prt
\newtheorem{proclaim}[thm]{\the\prt}
\theoremstyle{definition}

\newtheorem{definition}[thm]{Definition}

\def\eqn#1$$#2$${\begin{equation}\label#1#2\end{equation}}

\numberwithin{equation}{section}

\def\diam{\operatorname{diam}}

\def\epsilon{\varepsilon}

\def\er{\mathbb R}

\def\mir1{\mathcal L_1}

\def\phi{\varphi}

\newcommand{\yy}{\mathbb{Y}}

\newcommand{\dd}{\mathbb{D}}
\newcommand{\rr}{\mathbb{R}}
\newcommand{\partcurve}{\operatorname{part}}

\newtoks\by
\newtoks\paper
\newtoks\book
\newtoks\jour
\newtoks\yr
\newtoks\pages
\newtoks\vol
\newtoks\publ
\def\ota{{\hbox\vol{???}}}
\def\cLear{\by=\ota\paper=\ota\book=\ota\jour=\ota\yr=\ota
\pages=\ota\vol=\ota\publ=\ota}
\def\endpaper{\the\by, {\the\paper},
\textit{\the\jour} \textbf{\the\vol} (\the\yr), \the\pages.\cLear}
\def\endbook{\the\by, \textit{\the\book}, \the\publ.\cLear}
\def\endprep{\the\by, \textit{\the\paper}, \the\jour.\cLear}
\def\endyearprep{\the\by, \textit{\the\paper}, \the\jour, (\the\yr).\cLear}
\def\name#1#2{#2 #1}
\def\nom{ \rm no. }

\title{Extension of planar H\"older homeomorphisms}

\author[S. Hencl]{Stanislav Hencl}
\address{Charles University, Department of Mathematical Analysis Sokolovsk\'a 83, 186 00 Prague 8, Czech Republic}
\email{hencl@karlin.mff.cuni.cz}

\author[A. Koski]{Aleksis Koski}
\address{Department of Mathematics and Statistics, P.O. Box 68 (Pietari Kalmin katu 5), FI-00014 University of Helsinki, Finland}
\email{aleksis.koski@helsinki.fi}

\thanks{S. Hencl was supported by the grant GA\v CR P201/21-01976S.  A. Koski was supported by the ERC Advanced Grant number 834728 and by the Finnish Centre of Excellence in Randomness and Structures.}

\date{\today}

\begin{document}

\begin{abstract}
Let $\alpha\in (0,1)$. We show that any $\alpha$-H\"older homeomorphism from the unit circle in the plane to the plane can be extended to an $\alpha$-H\"older homeomorphism from the whole unit disc. 
\end{abstract}

\maketitle

\section{Introduction}

The well-known Kirszbraun extension theorem tells us that we can extend a Lipschitz mapping from any subset of $\er^n$ into $\er^n$ into a Lipschitz mapping of the whole space. However, things get more difficult once we ask our mapping to be injective as well. Let us denote by $\dd$ the unit disc in $\er^2$ and by $\partial \dd$ the unit circle. The classical Schoenflies theorem \cite[Theorem 10.4]{Mo} tells us that any continuous injective mapping $\varphi:\partial\dd\to\er^2$ can be extended to a homeomorphism from $\er^2$ onto $\er^2$. 

The assumption about injectivity is crucial in many applications as it corresponds e.g. to the "non-interpenetration of the matter" in models of Nonlinear Elasticity, see e.g. Ball \cite{Ba}. Similarly it appears naturally in Geometric Function Theory once we want our extension to be quasiconformal as done by Ahlfors \cite{Ah} or Tukia and V\"ais\"al\"a \cite{TV}. Requiring the extension to be both injective and Lipschitz complicates things even further. Such results were first obtained by Tukia in \cite{T} and \cite{T2}. 

Recently these problems became important in the solution of the so called Ball-Evans approximation problem, i.e. given a Sobolev $W^{1,p}$ homeomorphism the question is whether there exists a sequence of diffeomorphisms (or piecewise-linear homeomorphisms) that approximate it in the Sobolev norm. This problem was solved in the planar case for $1<p<\infty$ by Kovalev, Iwaniec and Onninen in \cite{IKO} and the remaining case $p=1$ by Hencl and Pratelli in \cite{HP}, see also Mora-Corral \cite{M} and Bellido and Mora-Corall \cite{BMC} for some initial results in this direction and Daneri and Pratelli \cite{DP} for some results about also approximating the inverse mapping. In many of these results the main crucial step is to establish some extension theorem from a given homeomorphism $\varphi:\partial\dd\to\er^2$ to the homeomorphism $h:\dd\to\er^2$ (or from $\varphi:\er^2\to \partial\dd$ to $h:\er^2\to \dd$ in \cite{IKO}) with some control of $\int |Dh|^p$ or of the Lipschitz constant, see e.g. Hencl and Pratelli \cite[Theorem 2.1]{HP} or Daneri and Pratelli \cite{DP2}. 
The extension of Sobolev homeomorphisms from the boundary attracted a lot of independent attention recently and we recommend e.g. \cite{HKO}, \cite{KKO} or \cite{KO} and references given there for further reading.  

The planar extension of Lipschitz homeomorphisms from the circle has been recently improved by Kovalev in \cite{K} and \cite{K2}, where he was able to obtain a sharp dependence of both the Lipschitz- and bilipschitz constants of the extension in terms of the constants for the original boundary map.  

The main aim of the present paper is to study the extension of $\alpha$-H\"older homeomorphisms from the unit circle $\partial \dd$. A mapping $\varphi$ is called $\alpha$-H\"older for $\alpha\in(0,1)$ if there is $C>0$ with
$$
|\varphi(x)-\varphi(y)|\leq C|x-y|^{\alpha}\text{ for every }x,y. 
$$ 
Our main result is the following extension theorem from the unit circle $\partial \dd$ in the plane.  

\prt{Theorem}
\begin{proclaim}\label{main}
Let $\alpha\in(0,1)$ and let $\varphi:\partial\dd\to\er^2$ be a homeomorphism which is $\alpha$-H\"older continuous. Then there is a homeomorphism $H:\overline{\dd}\to\er^2$ which is $\alpha$-H\"older and satisfies $H=\varphi$ on $\partial\dd$. 
\end{proclaim}

Let us note that it is not possible to use previously known extension procedures. One of the main reasons is that the Lipschitz-property is local and any locally Lipschitz $H$ is globally Lipschitz. However, a locally H\"older continuous map (on small scales) might not be globally H\"older continuous on bigger scales and thus we require new ideas. 

In our proof we cut $\dd$ into some dyadic parts and we define our $h$ on those parts separately. We define images of boundaries of those dyadic parts by somehow "copying"  the behavior of $\phi$ on the boundary nearby. 
Then we extend the mapping inside those images by using shortest curves inside with the help of some ideas from \cite{HP}. Since everything is dyadic in the domain we can sum the local 
estimates using a basic geometric series estimate and we obtain the desired global H\"older estimate at the end.


\section{Examples and technical lemma}

In our extension results we assume that we map a circle to a Jordan curve in the plane. The following simple example shows that the circle cannot be replaced by some non-Lipschitz domain even for a Lipschitz mapping $\varphi$. We use the following simple domains with an outer and inner cusp (see Fig \ref{domains})
$$
\begin{aligned}
\Omega_1&:=(-1,1)^2\cup\{[x,y]:\ x\in[1,2),\ |y|<(2-x)^2\}\text{ and }\\
\Omega_2&:=(-1,1)^2\setminus\{[x,y]:\ x\in[0,1], |y|<x^2\}.\\
\end{aligned}
$$

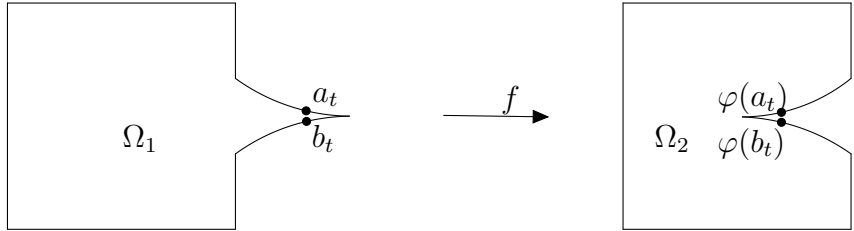
\begin{figure}
\vskip -50pt
{
\begin{tikzpicture}[line cap=round,line join=round,>=triangle 45,x=0.5cm,y=0.5cm]
\clip(-4.3,-11.319999999999997) rectangle (29.78,6.299999999999999);
\draw (0.0,1.0)-- (0.0,-5.0);
\draw (0.0,-5.0)-- (6.0,-5.0);
\draw (6.0,-5.0)-- (6.0,-3.0);
\draw (6.0,-1.0)-- (6.0,1.0);
\draw (6.0,1.0)-- (0.0,1.0);
\draw [shift={(9.063939393939394,3.1918181818181814)}] plot[domain=4.08120266372083:4.700074188255265,variable=\t]({1.0*5.1922118869664*cos(\t r)+-0.0*5.1922118869664*sin(\t r)},{0.0*5.1922118869664*cos(\t r)+1.0*5.1922118869664*sin(\t r)});
\draw [shift={(9.063939393939394,-7.191818181818181)}] plot[domain=1.5831111189243212:2.2019826434587566,variable=\t]({1.0*5.1922118869664*cos(\t r)+-0.0*5.1922118869664*sin(\t r)},{0.0*5.1922118869664*cos(\t r)+1.0*5.1922118869664*sin(\t r)});
\draw (16.2,1.0)-- (16.2,-5.0);
\draw (16.2,-5.0)-- (22.2,-5.0);
\draw (22.2,-5.0)-- (22.2,-3.0);
\draw (22.2,-1.0)-- (22.2,1.0);
\draw (22.2,1.0)-- (16.2,1.0);
\draw [shift={(19.321136363636363,2.5524999999999944)}] plot[domain=4.7165144105320955:5.393425299103003,variable=\t]({1.0*4.572538910362253*cos(\t r)+-0.0*4.572538910362253*sin(\t r)},{0.0*4.572538910362253*cos(\t r)+1.0*4.572538910362253*sin(\t r)});
\draw [shift={(19.134441416893736,-7.283160762942772)}] plot[domain=0.9495954761730744:1.53176005759885,variable=\t]({1.0*5.267173392595784*cos(\t r)+-0.0*5.267173392595784*sin(\t r)},{0.0*5.267173392595784*cos(\t r)+1.0*5.267173392595784*sin(\t r)});
\draw [->] (11.48,-1.98) -- (14.26,-2.02);
\draw (12.700000000000001,-0.9099999999999999) node[anchor=north west] {$f$};
\draw (2.74,-1.9599999999999999) node[anchor=north west] {$\Omega_1$};
\draw (16.74,-1.9599999999999999) node[anchor=north west] {$\Omega_2$};
\draw (7.74,-0.9599999999999999) node[anchor=north west] {$a_t$};
\draw (7.72,-1.9399999999999997) node[anchor=north west] {$b_t$};
\draw (18.4,-0.80399999999999998) node[anchor=north west] {$\varphi(a_t)$};
\draw (18.4,-2.0199999999999996) node[anchor=north west] {$\varphi(b_t)$};
\begin{scriptsize}
\draw [fill=black] (7.860251630130995,-1.8589441414765036) circle (1.5pt);
\draw [fill=black] (7.8713095233549275,-2.13843334876372) circle (1.5pt);
\draw [fill=black] (20.36902817370631,-1.898346530848396) circle (1.5pt);
\draw [fill=black] (20.36899040663466,-2.1627111401460937) circle (1.5pt);
\end{scriptsize}
\end{tikzpicture}
}
\vskip -100pt
\caption{Definition of $\Omega_1$, $\Omega_2$ and $\varphi$ between $\partial\Omega_1$ and $\partial\Omega_2$.}\label{domains}
\end{figure}

\prt{Example}
\begin{proclaim}
There is a Lipschitz homeomorphism $\varphi:\partial \Omega_1\to\partial\Omega_2$ which does not admit a Lipschitz homeomorphic extension $h: \overline{\Omega_1}\to\overline{\Omega_2}$. 
\end{proclaim}
\begin{proof}
We simply define $\varphi$ as the identity map on $\partial\Omega_1\cap \partial\Omega_2$, and a simple reflection across the line $x=1$ from the inner cusp to the outer cusp, i.e. we set 
$$
\varphi(x,y)=
\begin{cases}
[x,y]&\text{ for }[x,y]\in\partial \Omega_1\cap\partial (-1,1)^2,\\ 
[2-x,y]&\text{ for }[x,y]\in\partial \Omega_1\setminus\partial (-1,1)^2.\\
\end{cases}
$$
This mapping is clearly a homeomorphism from $\partial\Omega_1$ to $\partial\Omega_2$ and it is not difficult to find out that it is Lipschitz. 

Given small $t>0$ we consider two points (see Fig. \ref{domains})
$$
a_t=[2-t,t^2]\text{ and }b_t=[2-t,-t^2]
$$  
that are mapped to 
$$
\varphi(a_t)=[t,t^2]\text{ and }\varphi(b_t)=[t,-t^2]. 
$$
The segment $L_t:=\{[2-t,y]: |y|\leq t^2\}$ of length $2t^2$ has to be mapped by the extension to some continuum that connects $\varphi(a_t)$ and $\varphi(b_t)$ inside $\Omega_2$. However, the inner distance of $\varphi(a_t)$ and $\varphi(b_t)$ inside $\Omega_2$ is equal to $2\sqrt{t^2+t^4}$ and hence the Lipschitz constant of the extension needs to be at least
$$
\frac{2\sqrt{t^2+t^4}}{2t^2}=\frac{\sqrt{1+t^2}}{t}\text{ and this tends to }\infty\text{ as }t\to 0+.
$$
\end{proof}

We next give a simple characterization of curves that admit a H\"older parametrization, classifying all the Jordan curves for which our main theorem is applicable.  

\begin{definition} Let $p\in[1,\infty)$. A Jordan curve $\Gamma \subset \rr^2$ has finite $p$-content if there is a constant $C$ such that for every finite subdivision $x_1,x_2,\ldots, x_n \in \Gamma$ it holds that
\[\sum_{k=1}^{n-1} |x_{k+1} - x_k|^p \, \leq \, C.\]
The smallest such $C$ is denoted $|\Gamma|_p$.
\end{definition}

\begin{thm} Let $\alpha\in(0,1)$. A Jordan curve $\Gamma$ admits an $\alpha$-H\"older parametrization $\varphi: [0,1] \to \Gamma$ if and only if it has finite $\frac{1}{\alpha}$-content.
\end{thm}
\begin{proof} The 'only if' part of this statement is trivial, so we will only prove here that a curve with finite $\frac{1}{\alpha}$-content admits such a parametrization.

Let $p = \alpha^{-1}$. Note that any subcurve $\gamma \subset \Gamma$ also has finite $p$-content, and it is simple to verify that
\begin{equation}\label{eq:spliteq1}|\Gamma|_p \geq |\gamma|_p + |\Gamma \setminus \gamma|_p.\end{equation}
Hence given an initial parametrization $\varphi_0 : [0,1] \to \Gamma$, the map $F(t) = |\varphi_0([0,t])|_p$ defines an increasing surjection from $[0,1]$ to $[0,|\Gamma|_p]$. By setting
\[\varphi(t) = \varphi_0(F^{-1}(t))\]
we obtain a parametrization $\varphi: [0,|\Gamma|_p] \to \Gamma$ which satisfies
\begin{equation}\label{eq:spliteq2}|\varphi([0,t])|_p = t.\end{equation} Moreover, by the definition of the $p$-content $|\varphi([t_1,t_2])|_p$ and combining \eqref{eq:spliteq1} and \eqref{eq:spliteq2} we have
\[|\varphi(t_2) - \varphi(t_1)| \leq |\varphi([t_1,t_2])|_p^{\alpha} \leq \left(|\varphi([0,t_2])|_p-|\varphi([0,t_1])|_p\right)^{\alpha} = (t_2-t_1)^\alpha.\]
Hence $\varphi$ is $\alpha$-H\"older continuous. The map $t \to \varphi(t|\Gamma|_p)$ also gives such a parametrization on $[0,1]$ instead with H\"older-constant at most $|\Gamma|_p^\alpha$.
\end{proof}

\section{Proof of the H\"older extension result}

\begin{proof}[Proof of Theorem \ref{main}]
We are hence given an $\alpha$-H\"older parametrization $\varphi: \partial \dd \to \partial \yy$ of the boundary of a Jordan domain $\yy$ and we want to extend it to a H\"older-continuous homeomorphism $H: \dd \to \yy$. 
\vskip 5pt
\textbf{Step 1. Suitable decomposition of the domain and target.} 
To construct the extension we will employ a suitable decomposition of the target and domain sides. We will split the unit disk $\dd$ into dyadic regions $U_{n,k}$, each with diameter at most $2^{-n}$, and also split the domain side into appropriate regions $V_{n,k}$ depending on the behaviour of the boundary map $\varphi$. The aim is to define an extension $h : U_{n,k} \to V_{n,k}$ in each region and to show that it is H\"older-continuous with uniform constant $C$ in each region.

At a later stage of the proof, we will be able to use the dyadic structure of the sets $U_{n,k}$ to conclude that the extension $h$ will in fact be H\"older-continuous in the whole unit disk. The initial extension $h$ will not be a homeomorphism as it will map some open sets onto curves and map some of the interior of $\partial \dd$ onto the boundary $\partial \yy$, but it will define a monotone map (preimage of each point is connected) and we will take care in the construction in order for there to be an injectification process at the end where $h$ may be modified in an arbitrarily small way to give a homeomorphic extension $H: \dd \to \yy$ of $\varphi$. The fact that this modification is as small as we wish (on each dyadic level as well) is going to let us argue that $H$ inherits the H\"older-continuity estimates from $h$.

We start by constructing the dyadic sets $U_{n,k}$. We first split the unit circle $\partial \dd$ into dyadic arcs $I_{n,k}$, $n = 1,2,\ldots$ and $k = 1,\ldots,2^n$. The  amount of initial arcs or their position does not matter. For the purposes of simplifying the notation here, we may suppose that the arcs $I_{n,k}$ are instead intervals on the real line by flattening out the circle locally. The disk $\dd$ may be locally interpreted as the upper half space for this purpose.

We now let $T_{n,k}$ be the isosceles triangle in the upper half space with base $I_{n,k}$ and two equal angles of size $\pi/4 + 100^{-n}$ against the base. The apex vertex of $T_{n,k}$ is denoted $p_{n,k}$ and the endpoints of $I_{n,k}$ are denoted by $a_{n,k}$ and $b_{n,k}$. If we suppose that $I_{n+1,k'}$ and $I_{n+1,k'+1}$ are the two dyadic children of $I_{n,k}$, then we let $U_{n,k} = T_{n,k} \setminus \left(I_{n+1,k'} \cup I_{n+1,k'+1}\right)$ denote the region between each successive dyadic parent triangle and its children 
(see Fig. \ref{fig:domain1} for the idea how $U=U_{n,k}$ looks like).

We now wish to explain how the sets $V_{n,k}$ will be defined. Given the boundary map $\varphi$, we let $A_{n,k} = \varphi(a_{n,k})$ and $B_{n,k} = \varphi(b_{n,k})$. Moreover, let $C_{n,k}$ be the image of the middle point of $I_{n,k}$ under $\varphi$. We now may need to make a certain adjustment to this construction.

Let us call a point $\omega \in \partial \yy$ a \emph{good point} if there exists a line segment inside $\yy$ with endpoint $\omega$. We wish to assume that the dyadic image points given by $A_{n,k}$ are good points for all $n,k$. If this is not the case, we argue as follows. For every point $\omega \in \partial \yy$, it is possible to choose a good point $\omega' \in \partial \yy$ arbitrarily close to $\omega$ by considering a point $P \in \yy$ very close to $\omega$, considering the largest disk of center $P$ which lies entirely within $\yy$, and picking $\omega'$ as an intersection point of $\partial \yy$ with the boundary of this disk. This way of defining $\omega'$ guarantees that it is a good point. As $P$ approaches $\omega$, it can be shown that $\omega'$ approaches $\omega$ as well.

By appropriately using this observation inductively, we may replace each point $A_{n,k}$ with a good point $A'_{n,k}$ so that the preimages of $A'_{n,k}$ under $\varphi$ form another decomposition of the unit circle into intervals $I'_{n,k}$ which have length uniformly comparable to the length of $I_{n,k}$ for all $n,k$. If this is the case, it is not difficult to construct a self-map of the unit disk which is bilipschitz and maps each $I_{n,k}$ to $I'_{n,k}$. Since composition by bilipschitz-maps does not affect H\"older-continuity estimates, we may as well suppose that the points $A_{n,k}$ were good points to begin with and not deal with additional notation.

Let $\Gamma_{n,k}$ denote the unique shortest curve from $A_{n,k}$ to $B_{n,k}$ inside $\overline{\yy}$. Moreover, let $\Gamma_{n,k}^+$ and $\Gamma_{n,k}^{-}$ denote the corresponding curves for the dyadic children in this case. The curves $\Gamma_{n,k}, \Gamma_{n,k}^+,$ and $\Gamma_{n,k}^{-}$ may overlap each other and the boundary $\partial \yy$ but not cross each other. The region $V_{n,k}$ will be defined as a set bounded by these three curves, but let us first fix a parametrization for each of them.
\vskip 5pt
\textbf{Step 2. Fixing a parametrization.} Let us now define a parametrization of the curve $\Gamma_{n,k}$ which will be used later to define the values of the extension $h$ of $\varphi$ in part of the boundary of $U_{n,k}$. The main goal here is to define this parametrization as a H\"older-continuous map from the interval $I_{n,k}$ to $\Gamma_{n,k}$.

For purposes of easier presentation, let us drop subscripts for a moment and denote $I = I_{n,k} = [a,b]$. Let $\Gamma \in \overline{\yy}$ be the shortest curve between $\varphi(a)$ and $\varphi(b)$. We wish to define a parametrization $\tau: I \to \Gamma$ by ''projecting'' the parametrization given by $\varphi$ on the boundary. Given a point $z \in I$, we let $\tau(z)$ be the point on $\Gamma$ which is closest to $\varphi(z)$ with respect to the internal distance in $\yy$ (see Fig. \ref{define}). Such a point is determined uniquely since the connected components of $\Gamma$ which do not touch $\varphi(I)$ must be locally concave towards $\varphi(I)$ due to $\Gamma$ being a shortest curve (see e.g. \cite[Step 5 of the proof of Theorem 2.1]{HP} for the properties of the shortest curves). This defines the parametrization $\tau$ as a monotone map from $I$ to $\Gamma$, with loss of injectivity being possible since the preimage of a single point on $\Gamma$ may be an interval in $I$ 
(for example points between $z_1$ and $z_2$ on Fig. \ref{define} are all mapped to the same points $\tau(z)$). We will now show that $\tau$ inherits the H\"older-continuity of $\varphi$. The lack of injectivity of $\tau$ will not pose a problem later as we can eventually injectify this map with an arbitrarily small change to the H\"older-constant. 

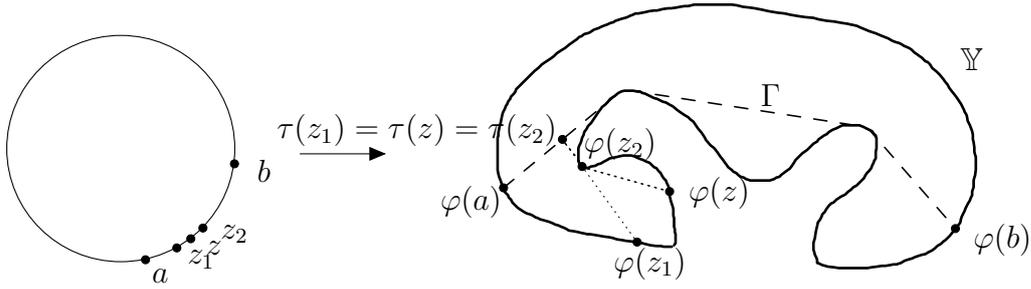
\begin{figure}
{
\vskip -60pt
\begin{tikzpicture}[line cap=round,line join=round,>=triangle 45,x=0.5cm,y=0.5cm]
\clip(-3.8302439024390234,-10.775121951219491) rectangle (29.06731707317073,6.0639024390243845);
\draw(0.4,-2.26) circle (1.5cm);
\draw [->] (5.12,-2.4) -- (7.38,-2.4);
\draw (0.9307317073170741,-5.155609756097549) node[anchor=north west] {$a$};
\draw (3.7014634146341474,-2.2678048780487727) node[anchor=north west] {$b$};
\draw (2.355121951219513,-4.433658536585355) node[anchor=north west] {$z$};
\draw [line width=1.0pt] (14.901463414634149,-4.862926829268281)-- (14.784390243902441,-4.862926829268281)-- (14.686829268292684,-4.862926829268281)-- (14.60878048780488,-4.84341463414633)-- (14.491707317073173,-4.84341463414633)-- (14.335609756097561,-4.823902439024378)-- (14.179512195121953,-4.7848780487804765)-- (14.042926829268294,-4.745853658536574)-- (13.867317073170733,-4.726341463414622)-- (13.691707317073172,-4.706829268292671)-- (13.594146341463418,-4.68731707317072)-- (13.457560975609757,-4.667804878048769)-- (13.340487804878052,-4.628780487804867)-- (13.242926829268296,-4.609268292682916)-- (13.086829268292686,-4.589756097560963)-- (12.98926829268293,-4.570243902439013)-- (12.872195121951222,-4.550731707317063)-- (12.755121951219515,-4.531219512195111)-- (12.67707317073171,-4.531219512195111)-- (12.599024390243905,-4.511707317073159)-- (12.501463414634149,-4.492195121951209)-- (12.423414634146342,-4.472682926829257)-- (12.325853658536587,-4.4531707317073055)-- (12.189268292682929,-4.433658536585355)-- (12.130731707317075,-4.394634146341452)-- (12.05268292682927,-4.355609756097549)-- (11.994146341463416,-4.336097560975599)-- (11.916097560975611,-4.297073170731697)-- (11.838048780487807,-4.277560975609745)-- (11.760000000000002,-4.238536585365844)-- (11.662439024390245,-4.19951219512194)-- (11.525853658536587,-4.16048780487804)-- (11.447804878048782,-4.1409756097560875)-- (11.369756097560977,-4.101951219512184)-- (11.311219512195123,-4.082439024390234)-- (11.25268292682927,-4.04341463414633)-- (11.174634146341466,-4.004390243902429)-- (11.09658536585366,-3.984878048780477)-- (11.018536585365856,-3.9458536585365755)-- (10.96,-3.9068292682926726)-- (10.901463414634149,-3.8678048780487697)-- (10.842926829268295,-3.8287804878048686)-- (10.764878048780488,-3.7507317073170627)-- (10.706341463414635,-3.692195121951209)-- (10.667317073170734,-3.6141463414634045)-- (10.608780487804879,-3.5556097560975513)-- (10.58926829268293,-3.4970731707316975)-- (10.511219512195124,-3.419024390243893)-- (10.491707317073173,-3.3604878048780398)-- (10.472195121951222,-3.301951219512186)-- (10.433170731707317,-3.243414634146332)-- (10.413658536585368,-3.165365853658528)-- (10.394146341463415,-3.106829268292674)-- (10.355121951219514,-3.009268292682918)-- (10.335609756097563,-2.9312195121951135)-- (10.316097560975612,-2.8336585365853573)-- (10.316097560975612,-2.7556097560975528)-- (10.316097560975612,-2.638536585365846)-- (10.316097560975612,-2.5214634146341384)-- (10.316097560975612,-2.3848780487804797)-- (10.316097560975612,-2.287317073170724)-- (10.316097560975612,-2.170243902439017)-- (10.316097560975612,-2.0531707317073096)-- (10.316097560975612,-1.975121951219505)-- (10.335609756097563,-1.9165853658536514)-- (10.335609756097563,-1.8385365853658469)-- (10.355121951219514,-1.7409756097560911)-- (10.374634146341464,-1.643414634146335)-- (10.394146341463415,-1.5848780487804814)-- (10.413658536585368,-1.4873170731707253)-- (10.433170731707317,-1.3897560975609695)-- (10.45268292682927,-1.3117073170731648)-- (10.511219512195124,-1.1946341463414578)-- (10.530731707317075,-1.136097560975604)-- (10.569756097560978,-1.0775609756097506)-- (10.58926829268293,-1.019024390243897)-- (10.647804878048783,-0.9409756097560923)-- (10.667317073170734,-0.8824390243902385)-- (10.745365853658539,-0.8043902439024339)-- (10.78439024390244,-0.7458536585365804)-- (10.803902439024393,-0.6873170731707267)-- (10.901463414634149,-0.5897560975609708)-- (10.96,-0.5312195121951172)-- (11.018536585365856,-0.47268292682926366)-- (11.07707317073171,-0.4141463414634102)-- (11.135609756097562,-0.3556097560975566)-- (11.23317073170732,-0.2580487804878006)-- (11.291707317073174,-0.19951219512194704)-- (11.38926829268293,-0.12146341463414229)-- (11.486829268292684,-0.02390243902438634)-- (11.54536585365854,0.01512195121951603)-- (11.603902439024392,0.054146341463418415)-- (11.662439024390245,0.09317073170732078)-- (11.7209756097561,0.13219512195122318)-- (11.799024390243904,0.19073170731707673)-- (11.877073170731709,0.26878048780488145)-- (11.935609756097563,0.3078048780487838)-- (11.994146341463416,0.3468292682926863)-- (12.091707317073173,0.4053658536585398)-- (12.20878048780488,0.5029268292682957)-- (12.267317073170735,0.5419512195121982)-- (12.345365853658539,0.6004878048780516)-- (12.40390243902439,0.6395121951219541)-- (12.462439024390246,0.6980487804878076)-- (12.5209756097561,0.7175609756097587)-- (12.599024390243905,0.7760975609756124)-- (12.67707317073171,0.7956097560975636)-- (12.735609756097562,0.8151219512195148)-- (12.794146341463415,0.8541463414634172)-- (12.85268292682927,0.8931707317073194)-- (12.911219512195123,0.9126829268292708)-- (12.969756097560976,0.9321951219512218)-- (13.047804878048783,0.9712195121951244)-- (13.145365853658538,1.0102439024390266)-- (13.203902439024393,1.0297560975609779)-- (13.281951219512196,1.06878048780488)-- (13.360000000000003,1.0882926829268313)-- (13.418536585365857,1.1273170731707336)-- (13.477073170731709,1.1468292682926848)-- (13.594146341463418,1.1858536585365873)-- (13.67219512195122,1.2053658536585385)-- (13.730731707317076,1.2248780487804898)-- (13.80878048780488,1.263902439024392)-- (13.867317073170733,1.2834146341463433)-- (13.945365853658538,1.3224390243902455)-- (14.023414634146345,1.3224390243902455)-- (14.1209756097561,1.3419512195121968)-- (14.218536585365857,1.361463414634148)-- (14.29658536585366,1.3809756097560992)-- (14.43317073170732,1.4004878048780505)-- (14.511219512195122,1.4004878048780505)-- (14.58926829268293,1.4200000000000017)-- (14.667317073170734,1.4395121951219527)-- (14.803902439024393,1.459024390243904)-- (14.8819512195122,1.459024390243904)-- (15.018536585365856,1.478536585365855)-- (15.116097560975614,1.4980487804878062)-- (15.233170731707318,1.5175609756097574)-- (15.330731707317076,1.5175609756097574)-- (15.40878048780488,1.5370731707317087)-- (15.486829268292684,1.55658536585366)-- (15.603902439024395,1.55658536585366)-- (15.760000000000003,1.55658536585366)-- (15.838048780487808,1.55658536585366)-- (15.916097560975611,1.55658536585366)-- (16.03317073170732,1.55658536585366)-- (16.111219512195124,1.55658536585366)-- (16.208780487804884,1.55658536585366)-- (16.36487804878049,1.55658536585366)-- (16.481951219512197,1.55658536585366)-- (16.618536585365856,1.5370731707317087)-- (16.696585365853664,1.5175609756097574)-- (16.774634146341466,1.4980487804878062)-- (16.950243902439027,1.478536585365855)-- (17.086829268292686,1.459024390243904)-- (17.281951219512198,1.4200000000000017)-- (17.45756097560976,1.3809756097560992)-- (17.691707317073174,1.3419512195121968)-- (17.867317073170735,1.3224390243902455)-- (18.042926829268296,1.2834146341463433)-- (18.277073170731708,1.244390243902441)-- (18.413658536585366,1.2248780487804898)-- (18.62829268292683,1.2053658536585385)-- (18.76487804878049,1.1663414634146363)-- (18.979512195121956,1.1273170731707336)-- (19.174634146341468,1.1078048780487826)-- (19.389268292682928,1.06878048780488)-- (19.52585365853659,1.049268292682929)-- (19.60390243902439,1.0297560975609779)-- (19.70146341463415,1.0102439024390266)-- (19.779512195121956,0.9907317073170754)-- (19.857560975609758,0.951707317073173)-- (19.955121951219517,0.9126829268292708)-- (20.091707317073173,0.8541463414634172)-- (20.20878048780488,0.8151219512195148)-- (20.267317073170737,0.7956097560975636)-- (20.403902439024392,0.73707317073171)-- (20.50146341463415,0.6980487804878076)-- (20.599024390243905,0.6590243902439052)-- (20.69658536585366,0.6004878048780516)-- (20.794146341463417,0.5614634146341493)-- (20.872195121951222,0.5419512195121982)-- (20.950243902439027,0.5029268292682957)-- (21.00878048780488,0.4639024390243934)-- (21.14536585365854,0.4053658536585398)-- (21.20390243902439,0.3663414634146374)-- (21.2819512195122,0.2882926829268327)-- (21.34048780487805,0.24926829268293033)-- (21.41853658536586,0.19073170731707673)-- (21.47707317073171,0.15170731707317434)-- (21.574634146341467,0.054146341463418415)-- (21.63317073170732,0.03463414634146722)-- (21.69170731707317,-0.02390243902438634)-- (21.750243902439028,-0.12146341463414229)-- (21.86731707317073,-0.21902439024389822)-- (21.94536585365854,-0.297073170731703)-- (22.042926829268293,-0.4141463414634102)-- (22.0819512195122,-0.47268292682926366)-- (22.160000000000004,-0.5702439024390196)-- (22.199024390243906,-0.6482926829268244)-- (22.25756097560976,-0.7263414634146291)-- (22.335609756097565,-0.8239024390243851)-- (22.39414634146342,-0.9019512195121899)-- (22.43317073170732,-0.9604878048780433)-- (22.47219512195122,-1.019024390243897)-- (22.511219512195122,-1.0775609756097506)-- (22.55024390243903,-1.1556097560975551)-- (22.58926829268293,-1.2336585365853598)-- (22.608780487804882,-1.2921951219512136)-- (22.64780487804878,-1.350731707317067)-- (22.667317073170732,-1.4092682926829208)-- (22.686829268292684,-1.526341463414628)-- (22.74536585365854,-1.643414634146335)-- (22.803902439024395,-1.7995121951219446)-- (22.823414634146342,-1.8970731707317003)-- (22.823414634146342,-2.0141463414634075)-- (22.842926829268293,-2.092195121951212)-- (22.862439024390245,-2.2287804878048707)-- (22.862439024390245,-2.345853658536578)-- (22.8819512195122,-2.423902439024382)-- (22.90146341463415,-2.482439024390236)-- (22.90146341463415,-2.560487804878041)-- (22.90146341463415,-2.6580487804877966)-- (22.90146341463415,-2.736097560975602)-- (22.90146341463415,-2.814146341463406)-- (22.8819512195122,-2.970243902439016)-- (22.862439024390245,-3.0287804878048696)-- (22.842926829268293,-3.106829268292674)-- (22.842926829268293,-3.20439024390243)-- (22.842926829268293,-3.2824390243902344)-- (22.823414634146342,-3.340975609756088)-- (22.803902439024395,-3.399512195121942)-- (22.803902439024395,-3.4970731707316975)-- (22.764878048780492,-3.5751219512195034)-- (22.74536585365854,-3.6336585365853553)-- (22.686829268292684,-3.731219512195112)-- (22.628292682926833,-3.848292682926819)-- (22.608780487804882,-3.9068292682926726)-- (22.56975609756098,-3.984878048780477)-- (22.530731707317074,-4.0629268292682825)-- (22.511219512195122,-4.121463414634136)-- (22.491707317073175,-4.2190243902438915)-- (22.452682926829272,-4.277560975609745)-- (22.39414634146342,-4.336097560975599)-- (22.355121951219513,-4.394634146341452)-- (22.29658536585366,-4.472682926829257)-- (22.218536585365857,-4.570243902439013)-- (22.10146341463415,-4.68731707317072)-- (22.042926829268293,-4.726341463414622)-- (21.984390243902443,-4.7848780487804765)-- (21.886829268292686,-4.84341463414633)-- (21.828292682926833,-4.901951219512184)-- (21.750243902439028,-4.960487804878038)-- (21.63317073170732,-5.019024390243891)-- (21.574634146341467,-5.058048780487794)-- (21.47707317073171,-5.077560975609745)-- (21.41853658536586,-5.097073170731695)-- (21.30146341463415,-5.136097560975597)-- (21.223414634146344,-5.1751219512194995)-- (21.14536585365854,-5.194634146341452)-- (21.028292682926832,-5.233658536585353)-- (20.93073170731708,-5.292195121951207)-- (20.81365853658537,-5.331219512195109)-- (20.735609756097563,-5.331219512195109)-- (20.638048780487807,-5.350731707317061)-- (20.5209756097561,-5.370243902439011)-- (20.423414634146344,-5.370243902439011)-- (20.306341463414636,-5.389756097560963)-- (20.15024390243903,-5.409268292682914)-- (20.05268292682927,-5.428780487804866)-- (19.955121951219517,-5.428780487804866)-- (19.87707317073171,-5.428780487804866)-- (19.74048780487805,-5.428780487804866)-- (19.66243902439025,-5.428780487804866)-- (19.52585365853659,-5.428780487804866)-- (19.44780487804878,-5.428780487804866)-- (19.291707317073175,-5.428780487804866)-- (19.174634146341468,-5.409268292682914)-- (18.999024390243907,-5.389756097560963)-- (18.94048780487805,-5.370243902439011)-- (18.862439024390248,-5.350731707317061)-- (18.78439024390244,-5.311707317073159)-- (18.725853658536586,-5.292195121951207)-- (18.725853658536586,-5.214146341463403)-- (18.70634146341464,-5.155609756097549)-- (18.667317073170732,-5.077560975609745)-- (18.647804878048785,-4.999512195121939)-- (18.62829268292683,-4.940975609756086)-- (18.62829268292683,-4.84341463414633)-- (18.62829268292683,-4.765365853658524)-- (18.667317073170732,-4.628780487804867)-- (18.686829268292687,-4.492195121951209)-- (18.70634146341464,-4.394634146341452)-- (18.725853658536586,-4.316585365853648)-- (18.74536585365854,-4.2190243902438915)-- (18.803902439024395,-4.1409756097560875)-- (18.842926829268293,-4.0629268292682825)-- (18.862439024390248,-4.004390243902429)-- (18.920975609756102,-3.9458536585365755)-- (18.979512195121956,-3.848292682926819)-- (19.03804878048781,-3.770243902439015)-- (19.096585365853663,-3.711707317073161)-- (19.155121951219517,-3.692195121951209)-- (19.233170731707318,-3.6141463414634045)-- (19.35024390243903,-3.5360975609756005)-- (19.467317073170737,-3.4775609756097468)-- (19.545365853658538,-3.399512195121942)-- (19.60390243902439,-3.3604878048780398)-- (19.681951219512197,-3.2824390243902344)-- (19.76,-3.243414634146332)-- (19.818536585365855,-3.184878048780479)-- (19.87707317073171,-3.1263414634146254)-- (19.955121951219517,-3.0482926829268204)-- (19.994146341463416,-2.970243902439016)-- (20.03317073170732,-2.911707317073162)-- (20.091707317073173,-2.8531707317073085)-- (20.15024390243903,-2.775121951219504)-- (20.169756097560978,-2.7165853658536503)-- (20.20878048780488,-2.638536585365846)-- (20.22829268292683,-2.579999999999992)-- (20.267317073170737,-2.5214634146341384)-- (20.267317073170737,-2.4434146341463334)-- (20.267317073170737,-2.345853658536578)-- (20.286829268292685,-2.248292682926822)-- (20.286829268292685,-2.1507317073170658)-- (20.286829268292685,-2.0726829268292613)-- (20.286829268292685,-1.9946341463414565)-- (20.267317073170737,-1.9165853658536514)-- (20.20878048780488,-1.8580487804877979)-- (20.169756097560978,-1.7995121951219446)-- (20.091707317073173,-1.7604878048780421)-- (20.01365853658537,-1.7214634146341397)-- (19.935609756097563,-1.6824390243902372)-- (19.83804878048781,-1.643414634146335)-- (19.70146341463415,-1.643414634146335)-- (19.58439024390244,-1.643414634146335)-- (19.486829268292688,-1.643414634146335)-- (19.389268292682928,-1.643414634146335)-- (19.311219512195127,-1.6824390243902372)-- (19.233170731707318,-1.7214634146341397)-- (19.135609756097566,-1.7604878048780421)-- (19.03804878048781,-1.8190243902438956)-- (18.979512195121956,-1.8775609756097493)-- (18.901463414634147,-1.9165853658536514)-- (18.76487804878049,-2.0141463414634075)-- (18.725853658536586,-2.0726829268292613)-- (18.62829268292683,-2.131219512195115)-- (18.530731707317077,-2.2287804878048707)-- (18.47219512195122,-2.287317073170724)-- (18.39414634146342,-2.345853658536578)-- (18.335609756097565,-2.3848780487804797)-- (18.277073170731708,-2.4434146341463334)-- (18.218536585365857,-2.501951219512187)-- (18.160000000000004,-2.560487804878041)-- (18.101463414634146,-2.6580487804877966)-- (18.062439024390248,-2.7165853658536503)-- (18.003902439024394,-2.794634146341455)-- (17.94536585365854,-2.8336585365853573)-- (17.867317073170735,-2.892195121951211)-- (17.789268292682927,-2.9507317073170642)-- (17.730731707317076,-3.009268292682918)-- (17.65268292682927,-3.0287804878048696)-- (17.594146341463418,-3.0678048780487717)-- (17.535609756097564,-3.087317073170723)-- (17.47707317073171,-3.106829268292674)-- (17.379512195121954,-3.106829268292674)-- (17.30146341463415,-3.106829268292674)-- (17.203902439024393,-3.106829268292674)-- (17.106341463414637,-3.106829268292674)-- (16.96975609756098,-3.106829268292674)-- (16.85268292682927,-3.0678048780487717)-- (16.794146341463417,-3.0287804878048696)-- (16.716097560975612,-3.009268292682918)-- (16.657560975609755,-2.970243902439016)-- (16.599024390243905,-2.911707317073162)-- (16.520975609756096,-2.8726829268292597)-- (16.462439024390246,-2.814146341463406)-- (16.384390243902445,-2.7556097560975528)-- (16.345365853658535,-2.6970731707316995)-- (16.286829268292685,-2.619024390243894)-- (16.247804878048786,-2.560487804878041)-- (16.169756097560978,-2.462926829268285)-- (16.091707317073173,-2.345853658536578)-- (16.013658536585368,-2.248292682926822)-- (15.916097560975611,-2.170243902439017)-- (15.799024390243904,-2.033658536585359)-- (15.760000000000003,-1.955609756097554)-- (15.720975609756099,-1.8970731707317003)-- (15.681951219512197,-1.8385365853658469)-- (15.584390243902439,-1.7214634146341397)-- (15.506341463414637,-1.6043902439024325)-- (15.467317073170735,-1.545853658536579)-- (15.38926829268293,-1.4873170731707253)-- (15.330731707317076,-1.428780487804872)-- (15.233170731707318,-1.3702439024390183)-- (15.116097560975614,-1.2921951219512136)-- (15.038048780487808,-1.2336585365853598)-- (14.9209756097561,-1.136097560975604)-- (14.842926829268295,-1.0970731707317016)-- (14.745365853658539,-1.0385365853658481)-- (14.667317073170734,-1.019024390243897)-- (14.58926829268293,-0.9799999999999945)-- (14.511219512195122,-0.9409756097560923)-- (14.43317073170732,-0.9214634146341409)-- (14.277073170731708,-0.8434146341463362)-- (14.199024390243904,-0.8043902439024339)-- (14.1209756097561,-0.7848780487804826)-- (14.062439024390246,-0.7458536585365804)-- (13.98439024390244,-0.7458536585365804)-- (13.886829268292683,-0.7263414634146291)-- (13.80878048780488,-0.7263414634146291)-- (13.730731707317076,-0.7263414634146291)-- (13.652682926829272,-0.7263414634146291)-- (13.574634146341467,-0.7458536585365804)-- (13.49658536585366,-0.7848780487804826)-- (13.399024390243904,-0.8629268292682873)-- (13.340487804878052,-0.9019512195121899)-- (13.281951219512196,-0.9409756097560923)-- (13.203902439024393,-1.0385365853658481)-- (13.145365853658538,-1.0970731707317016)-- (13.067317073170733,-1.1556097560975551)-- (13.028292682926832,-1.214146341463409)-- (12.969756097560976,-1.2921951219512136)-- (12.930731707317076,-1.350731707317067)-- (12.85268292682927,-1.4678048780487742)-- (12.813658536585368,-1.526341463414628)-- (12.774634146341466,-1.5848780487804814)-- (12.696585365853661,-1.662926829268286)-- (12.657560975609758,-1.7409756097560911)-- (12.599024390243905,-1.8190243902438956)-- (12.579512195121952,-1.8970731707317003)-- (12.540487804878051,-1.975121951219505)-- (12.5209756097561,-2.0531707317073096)-- (12.501463414634149,-2.1117073170731633)-- (12.462439024390246,-2.1897560975609682)-- (12.442926829268297,-2.248292682926822)-- (12.442926829268297,-2.345853658536578)-- (12.442926829268297,-2.423902439024382)-- (12.442926829268297,-2.5214634146341384)-- (12.442926829268297,-2.5995121951219433)-- (12.481951219512197,-2.677560975609748)-- (12.501463414634149,-2.736097560975602)-- (12.560000000000002,-2.7556097560975528)-- (12.618536585365856,-2.775121951219504)-- (12.696585365853661,-2.775121951219504)-- (12.774634146341466,-2.736097560975602)-- (12.891707317073172,-2.7165853658536503)-- (12.950243902439025,-2.677560975609748)-- (13.028292682926832,-2.6580487804877966)-- (13.125853658536586,-2.5995121951219433)-- (13.203902439024393,-2.560487804878041)-- (13.262439024390247,-2.5214634146341384)-- (13.360000000000003,-2.501951219512187)-- (13.457560975609757,-2.482439024390236)-- (13.535609756097564,-2.462926829268285)-- (13.633170731707319,-2.4434146341463334)-- (13.730731707317076,-2.4434146341463334)-- (13.828292682926831,-2.4434146341463334)-- (13.906341463414638,-2.462926829268285)-- (13.98439024390244,-2.501951219512187)-- (14.081951219512199,-2.5409756097560896)-- (14.140487804878052,-2.560487804878041)-- (14.218536585365857,-2.5995121951219433)-- (14.316097560975612,-2.638536585365846)-- (14.374634146341466,-2.6970731707316995)-- (14.43317073170732,-2.7556097560975528)-- (14.491707317073173,-2.775121951219504)-- (14.550243902439027,-2.8531707317073085)-- (14.60878048780488,-2.911707317073162)-- (14.647804878048781,-2.9897560975609667)-- (14.686829268292684,-3.0482926829268204)-- (14.745365853658539,-3.106829268292674)-- (14.764878048780488,-3.165365853658528)-- (14.803902439024393,-3.243414634146332)-- (14.823414634146342,-3.301951219512186)-- (14.842926829268295,-3.3604878048780398)-- (14.842926829268295,-3.438536585365844)-- (14.842926829268295,-3.5165853658536497)-- (14.862439024390245,-3.594634146341454)-- (14.8819512195122,-3.6726829268292582)-- (14.901463414634149,-3.731219512195112)-- (14.9209756097561,-3.7897560975609657)-- (14.94048780487805,-3.8678048780487697)-- (14.94048780487805,-3.9653658536585263)-- (14.94048780487805,-4.04341463414633)-- (14.960000000000003,-4.101951219512184)-- (14.960000000000003,-4.17999999999999)-- (14.960000000000003,-4.258048780487794)-- (14.960000000000003,-4.336097560975599)-- (14.979512195121954,-4.394634146341452)-- (14.979512195121954,-4.472682926829257)-- (14.999024390243903,-4.531219512195111)-- (14.999024390243903,-4.609268292682916)-- (14.999024390243903,-4.68731707317072)-- (14.999024390243903,-4.765365853658524)-- (14.960000000000003,-4.823902439024378);
\draw (22.199024390243903,0.6785365853658566) node[anchor=north west] {$\mathbb{Y}$};
\draw (8.520975609756099,-2.9702439024390155) node[anchor=north west] {$\varphi(a)$};
\draw (22.550243902439025,-3.9653658536585263) node[anchor=north west] {$\varphi(b)$};
\draw (15.057560975609757,-2.7360975609756015) node[anchor=north west] {$\varphi(z)$};
\draw [line width=0.6000000000000002pt,dash pattern=on 5pt off 5pt] (10.472195121951222,-3.301951219512186)-- (13.522665080309345,-0.7718381915526407)-- (13.836859012492566,-0.7263414634146291)-- (19.719071980963722,-1.643414634146335)-- (20.099922207477235,-1.7645952500800728)-- (22.36376149727722,-4.38167482725489);
\draw (16.950243902439027,-0.25804878048780044) node[anchor=north west] {$\Gamma$};
\draw [line width=0.6000000000000002pt,dotted] (14.842926829268293,-3.399512195121942)-- (12.617251635930996,-2.7746936347412174)-- (12.014891982569996,-2.0224115162775083);
\draw (4.220975609756098,-1.0580487804877992) node[anchor=north west] {$\tau(z_1)=\tau(z)=\tau(z_2)$};
\draw [dotted] (12.021848377897557,-2.016641760572269)-- (13.987245687091018,-4.739666864961321);
\draw (1.8673170731707327,-4.667804878048769) node[anchor=north west] {$z_1$};
\draw (2.7453658536585377,-3.9848780487804776) node[anchor=north west] {$z_2$};
\draw (13.086829268292684,-4.589756097560964) node[anchor=north west] {$\varphi(z_1)$};
\draw (12.301463414634147,-1.4214634146341397) node[anchor=north west] {$\varphi(z_2)$};
\begin{scriptsize}
\draw [fill=black] (1.0506510866749004,-5.208262736495643) circle (1.5pt);
\draw [fill=black] (3.392009492538974,-2.6643256070998618) circle (1.5pt);
\draw [fill=black] (2.2427360880089666,-4.651636199330787) circle (1.5pt);
\draw [fill=black] (10.472195121951222,-3.301951219512186) circle (1.5pt);
\draw [fill=black] (22.36376149727722,-4.38167482725489) circle (1.5pt);
\draw [fill=black] (14.842926829268293,-3.399512195121942) circle (1.5pt);
\draw [fill=black] (12.021848377897557,-2.016641760572269) circle (1.5pt);
\draw [fill=black] (1.8742385926147387,-4.894809399566753) circle (1.5pt);
\draw [fill=black] (2.561945688124724,-4.367508206769054) circle (1.5pt);
\draw [fill=black] (13.987245687091018,-4.739666864961321) circle (1.5pt);
\draw [fill=black] (12.540909940964609,-2.735792875893081) circle (1.5pt);
\end{scriptsize}
\end{tikzpicture}
\vskip -70pt	
}
\caption{Definition of the curve $\Gamma$ and its parametrization $\tau$. The whole part between $z_1$ and $z_2$ is mapped to the same point $\tau(z)$.}\label{define}
\end{figure}

Let $x,y \in I$ and let $\beta_{xy}$ denote the part of $\varphi(I)$ between $\varphi(x)$ and $\varphi(y)$. It will be enough to show that the inequality
\begin{equation}\label{eq:tau_ineq}
|\tau(x)-\tau(y)| \leq C \diam(\beta_{xy})
\end{equation} holds, since the diameter of $\beta_{xy}$ satisfies $\diam(\beta_{xy}) \leq C|x-y|^\alpha$ due to the H\"older-continuity of $\varphi$ on all subintervals.

We may suppose that the shortest curves $\gamma_x$ and $\gamma_y$ which connect $\varphi(x)$ to $\tau(x)$ and $\varphi(y)$ to $\tau(y)$ respectively within $\overline{\yy}$ are in fact straight line segments. This is due to the following fact. Let $I_x$ denote the preimage of $\tau(x)$ under $\tau$, which is either a single point or a closed interval containing $x$. Since the closest point on $\Gamma$ w.r.t. the internal distance from $\varphi(x')$ is $\tau(x)$ for all $x' \in I_x$, by geometry we see that if $x'$ is one of the endpoints of $I_x$ such a shortest curve must be a straight line segment between $\varphi(x')$ and $\tau(x)$ (for example on Fig. \ref{define} our $I_z$ is between $z_1$ and $z_2$ 
and shortest line between $\phi(z_1)$ and $\tau(z_1)$ or $\phi(z_2)$ and $\tau(z_2)$ is a segment). Let $x'$ be the endpoint of $I_x$ which is closest to $y$. Since $\beta_{x'y} \subset \beta_{xy}$, it is enough to prove \eqref{eq:tau_ineq} for $x'$ instead of $x$ because $\tau(x') = \tau(x)$. Hence we suppose that $\gamma_x$ and $\gamma_y$ are straight line segments.

Let $\Gamma'$ denote the part of $\Gamma$ between $\tau(x)$ and $\tau(y)$. We claim that it is enough to prove \eqref{eq:tau_ineq} when the interior of $\Gamma'$ does not touch the curve $\varphi(I)$. Indeed, if the interior does touch $\varphi(I)$ then it must be at a point on $\beta_{xy}$ as the part of $\varphi(I)$ not equal to $\beta_{xy}$ cannot cross $\gamma_x,\gamma_y$ or $\Gamma$ and hence has no access to $\Gamma'$. Supposing that \eqref{eq:tau_ineq} has been proven when the interior of $\Gamma'$ does not intersect $\beta_{xy}$, we explain now how the general case would follow. For this, note that if $u,v \in \Gamma' \cap \beta_{xy}$, then trivially $|u-v| \leq \diam(\beta_{xy})$. Hence by picking $u,v$ on $\Gamma'$ as close to the endpoints $\tau(x)$ and $\tau(y)$ as possible (with respect to parametrization), we may split $\Gamma'$ into three parts: One part from $\tau(x)$ to $u$, one part from $u$ to $v$, and one part from $v$ to $\tau(y)$. For clarification, it is possible for any of these parts to be a singleton as well. If $\tau(x)$ and $u$ are not equal, by definition of $u$ there cannot be a point on $\beta_{xy}$ intersecting the part of $\Gamma'$ between $\tau(x)$ and $u$. But now by assumption we can apply \eqref{eq:tau_ineq} in this part and obtain that $|\tau(x) - u| \leq C \diam(\beta_{xy})$. Similarly $|v - \tau(y)| \leq C \diam(\beta_{xy})$, and combining these estimates with the one for $|u-v|$ gives \eqref{eq:tau_ineq} while only increasing the constant by a factor $3$.

Hence we may suppose that the interior of $\Gamma'$ does not intersect $\beta_{xy}$. Since $\Gamma'$ is a shortest curve and does not intersect $\gamma_x$ or $\gamma_y$ beside the mutual endpoints, this means that $\Gamma'$ is locally concave towards the interior of the region $\Omega$ bounded by the four curves $\beta_{xy}, \gamma_x, \Gamma',$ and $\gamma_y$ (see e.g. \cite[Step 5 of the proof of Theorem 2.1]{HP}). For the sake of contradiction in proving \eqref{eq:tau_ineq}, let us suppose that $|\tau(x)-\tau(y)| > 100 \diam(\beta_{xy})$. Now if both $|\tau(x) - \varphi(x)|$ and $|\tau(y) - \varphi(y)|$ are smaller than $40 \diam(\beta_{xy})$, then
\[|\tau(x)-\tau(y)| \leq |\tau(x) - \varphi(x)| + |\varphi(x) - \varphi(y)| + |\tau(y) - \varphi(y)| \leq 81 \diam(\beta_{xy}),\]
a contradiction. Hence we may suppose without loss of generality that $|\tau(x) - \varphi(x)| > 40 \diam(\beta_{xy})$. 

Let us cover $\beta_{xy}$ by a ball $B$ of radius $2 \diam(\beta_{xy})$ (exact positioning does not matter here). Hence the line segment $\gamma_x$ starts inside of $B$ and ends outside of $B$. We now note that $\gamma_y$ must also do the same, as otherwise $\gamma_y$ would have no intersection point with $\partial B$ and $\Gamma'$ would have to intersect $\partial B$ somewhere. This would give a contradiction with the definition of $\tau(x)$: To connect $\phi(x)$ with $\Gamma$ it would be shorter to go from $x$ to a point in $\gamma_x\cap \partial B$ and then along $\partial B$ to the point in $\Gamma'\cap\partial B$ and the total length would be shorter than $40 \diam(\beta_{xy})$, giving us a contradiction with the definition of $\tau(x)$ and $|\tau(x) - \varphi(x)| > 40 \diam(\beta_{xy})$. 

Consider now a line $\ell$ picked so that it passes through $\tau(x)$ but the sets $\gamma_x, \gamma_y,$ and $B$ all lie on one (closed) side of $\ell$ (for example in Fig. \ref{fig:contradiction} our $\ell$ could be between $\tau(x)$ and $\tau(y)$ on the first picture and between $\tau(x)$ and $P$ on the second picture). We call this side the 'good side' just to have a name for it. It is not difficult to see that also $\Gamma'$ lies inside this good side of $\ell$ since $\Gamma'$ is concave towards $\beta_{xy}$ by the property of minimal curves. Of course it could happen that $\Gamma'$ contains both $\tau(x)$ and $\tau(y)$ and $\Gamma'$ is a straight segment and then it lies inside $\gamma$. 

\begin{figure}[htbp]
\centering
\includegraphics[width=0.7\textwidth]
{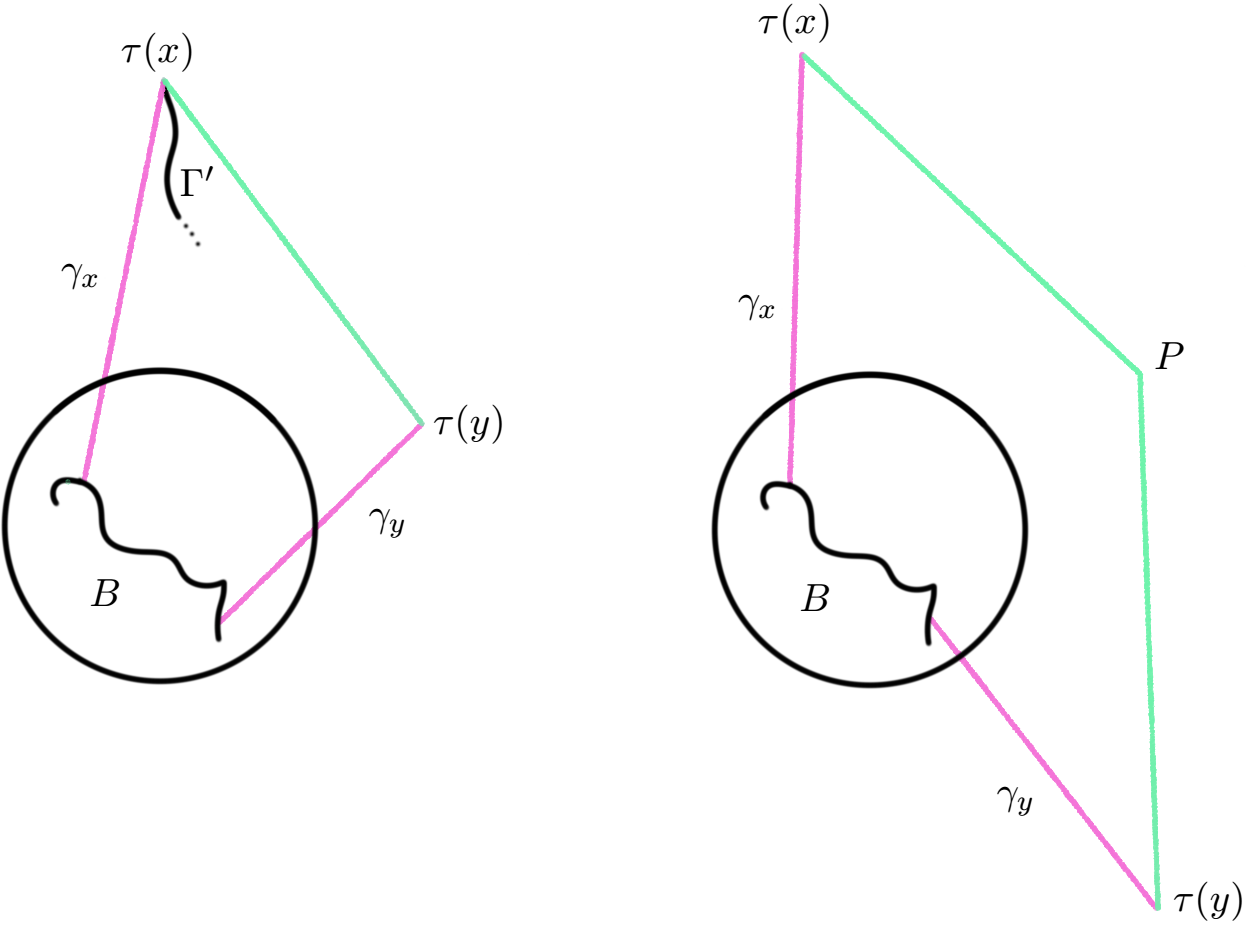}
\caption{Two situations which both lead to a contradiction.}
\label{fig:contradiction}
\end{figure}
Thus any line that passes through $\tau(x)$ (or $\tau(y)$) and contains the sets $\gamma_x, \gamma_y,$ and $B$ on one side of it automatically also contains $\Gamma'$ on that side, we call this property $(\ast)$. Now the claim \eqref{eq:tau_ineq} will be fairly easy to prove. If the line segment $L_{xy}$ from $\tau(x)$ to $\tau(y)$ does not intersect $B$, then due to property $(\ast)$ the curve $\Gamma'$ must lie within a region bounded by $L_{xy}, \gamma_x, \gamma_y$, and the part of $\partial B$ between $\gamma_x$ and $\gamma_y$. Hence $L_{xy}$ must form an angle greater or equal to $\pi/2$ with respect to both $\gamma_x$ and $\gamma_y$ due to the fact that $\tau(x)$ and $\tau(y)$ are the closest points on $\Gamma'$ from points on $\gamma_x$ and $\gamma_y$, as otherwise we could find closer points on $\Gamma'$ 
(for example in the situation of Fig. \ref{fig:contradiction} on the left close to $\tau(x)$ we could slightly cut the corner and we would go from $\gamma_x$ to $\Gamma'$ in a shorter way). This automatically shows \eqref{eq:tau_ineq} as then $|\tau(x) - \tau(y)|$ must be less than the diameter of $B$. If the line segment from $\tau(x)$ to $\tau(y)$ does intersect $B$, then we can find another point $P$ so that the segments $\tau(x) P$ and $P \tau(y)$ form acute angles with $\gamma_x$ and $\gamma_y$ respectively and by property $(\ast)$ they must contain $\Gamma'$ on the side with acute angles, see Fig. \ref{fig:contradiction}. This is a contradiction as then there must be points on $\Gamma'$ besides $\tau(x)$ which are closer to points on $\gamma_x$ than $\tau(x)$ is (in the situation of Fig. \ref{fig:contradiction} on the right close to $\tau(x)$ we could again slightly cut the corner). This proves the claim and hence the H\"older-continuity of $\tau$.
\vskip 5pt
\textbf{Step 3. The initial case.} We detail here the first step of the construction of the extension $h$ by giving an initial set $T \subset \dd$ in which we define the map $h$ on. The rest of the construction will be a repetition of this idea with respect to each dyadic interval.
\begin{figure}[htbp]
\centering
\includegraphics[width=0.9\textwidth]
{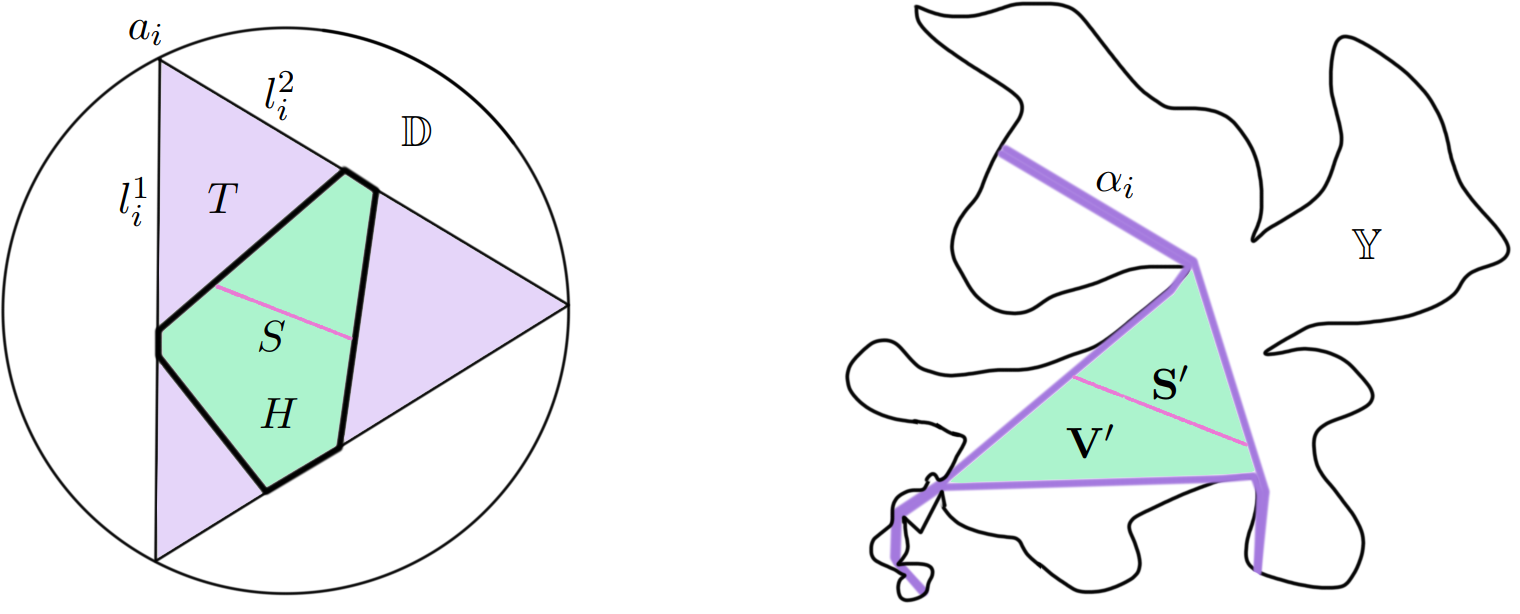}
\caption{The domain and target for the initial case.}
\label{fig:initial}
\end{figure}
Let us start the dyadic decomposition of $\partial \dd$ by fixing three initial arcs $I_{0,1}, I_{0,2}$ and $I_{0,3}$ of length $2\pi/3$ whose union is the whole circle. 
Hence we have an equilateral triangle $T$ in $\dd$ with vertices $a_i, i = 1,2,3$ given by the endpoints of these three arcs.

On the target side there are corresponding curves $\Gamma_{0,i}$, $i = 1,2,3$ given by the shortest curves in $\overline{\yy}$ between the image points $\varphi(a_i), i = 1,2,3$. These curves
may partially overlap, as the two curves $\Gamma_{0,i}$ which share a mutual endpoint $\varphi(a_j)$ may have a common part given by a curve $\alpha_j$ (see e.g. $\alpha_i$ on Fig \ref{fig:initial}). 
In the case where there is no overlap we simply consider $\alpha_j$ as the singleton $\varphi(a_j)$. However, each curve $\Gamma_{0,i}$ has some part which does not overlap with any of the others, and these three curves together define a Jordan curve that bounds a domain $\mathbf{V}'$. The particular thing to note about $\mathbf{V}'$ is that due to the fact that its boundary consists of three nonintersecting shortest curves, these three curves must be concave towards the interior of $\mathbf{V}'$ (again by \cite[Step 5 of the proof of Theorem 2.1]{HP}), see Fig. \ref{fig:initial} and also Fig. \ref{fig:shortest} for possible shapes of $\mathbf{V}'$ . We let $V$ denote the union of $\mathbf{V}'$ and the common parts $\alpha_j, j = 1,2,3$. Using the projected parametrization $\tau$ defined in previous step, we may define a boundary map $\tau' : \partial T \to \partial V$ which is monotone on each of the sides of $T$ and is H\"older-continuous with estimates in terms of the original boundary map $\varphi$.

We now aim to extend the map $\tau'$ to the interior of the triangle $T$ to give a map $h_0: T \to V$. The preimages of the curves $\alpha_i \subset V, i = 1,2,3$ are fairly easy to choose. For each $i$, we look at the two preimages $l^1_i$ and $l^2_i$ of $\alpha_i$ under $\tau'$, which are line segments starting from $a_i$ on the two sides of $\partial T$, see  Fig \ref{fig:initial}. Now for each point $P \in \alpha_i$, we look at the preimages $p_1 \in l^1_i$ and $p_2 \in l^2_i$ under $\tau'$ and connect $p_1$ and $p_2$ via a line segment $L$. Then for each $z \in L$ we simply define that $h_0(z) = P$. This effectively defines the map $h_0$ on the triangle bounded by $l^1_i$ and $l^2_i$. The map is H\"older-continuous as it is simply a linear interpolation of the two H\"older-continuous maps (given by $\tau'$) on $l^1_i$ and $l^2_i$.

This leaves a convex hexagonal region $H \subset T$ on the interior of which the map $h_0$ has not been defined yet. The boundary values of $h_0$ on $\partial H$ are already determined, and three of the sides $S_1,S_3,$ and $S_5$ of $H$ are mapped to single points under $h_0$ while the three other sides are mapped to concave curves (the boundary of $\mathbf{V}'$). To define $h_0$, we will employ a direct version of the shortest curve extension method. For clarity, let us denote by $\varphi_0 = h_0 \vert_{\partial H}$ the already defined boundary values of $h_0$.

Let us start by fixing a ''horizontal'' direction given by a line $\ell$, and we consider the ''vertical'' direction to be the one orthogonal to $\ell$. There are a few additional requirements we are able to impose. First, we may assume that any two points in $H$ can be connected by a curve consisting of at most five line segments that are either horizontal or vertical. This is due to the fact that $H$ is convex and the practical idea here is that if $H$ is very thin then we should choose the vertical direction to be parallel to the ''direction of thinness'' of $H$. 

Moreover, we may suppose that the angle between the horizontal direction $\ell$ and all of the sides of $H$ is controlled from below by some absolute constant $c > 0$. This is not immediately possible if $H$ is very thin and it has two sides which are close to being parallel to the horizontal direction we need to impose for the previous requirement, see Fig.  \ref{fig:adjust}. But in this case we make a small adjustment to these sides to replace them with two line segments that have a large angle towards the horizontal direction. Such an adjustment can be realized as a bilipschitz transformation of $H$ with absolute control on the constant so it does not affect future estimates, although it does transform the hexagon $H$ into an octagon. But we will later only use the fact that $H$ has a fixed finite amount of sides so this does not matter.

\begin{figure}[htbp]
\centering
\includegraphics[width=0.60\textwidth]
{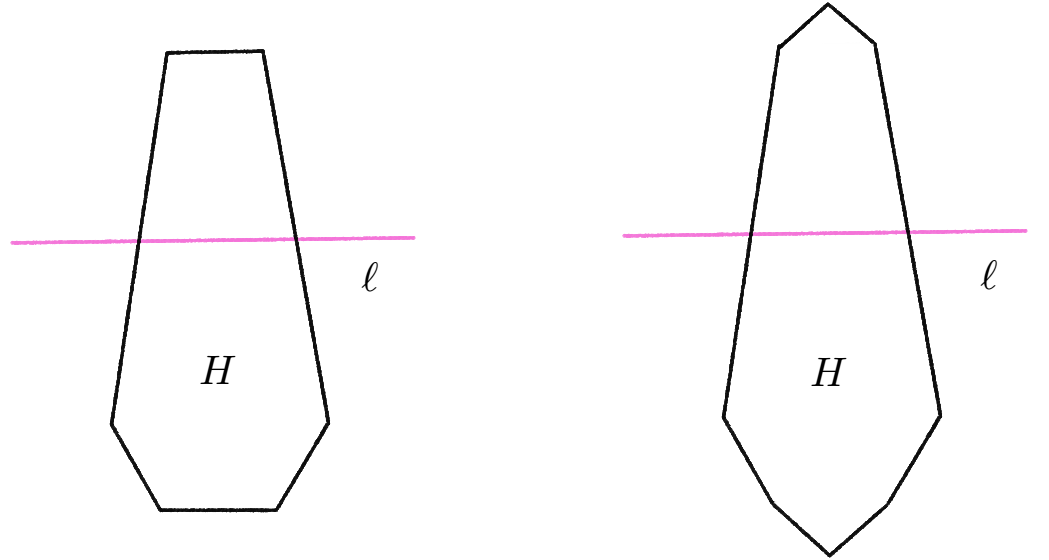}
\caption{Adjusting the hexagon.}
\label{fig:adjust}
\end{figure}

With these assumptions $H$ may be written as the union of line segments parallel to $\ell$, the first and last of which consist of simply a single vertex of $H$. We call these segments horizontal segments on $H$. For each such segment $S$ having endpoints $s_1$ and $s_2$, we define the map $h_0$ on $S$ by defining $h_0$ to map $S$ onto the shortest curve $\mathbf{S}'$ in $\mathbf{V}'$ between the points $\varphi_0(s_1)$ and $\varphi_0(s_2)$ on $\partial \mathbf{V}'$, parametrized at constant speed, see Fig. \ref{fig:shortest}. This defines a monotone map $h_0 : H \to \mathbf{V}'$, but it can be injectified with an arbitrarily small modification. Hence it remains to verify the H\"older-continuity of such a map $h_0$. It is enough to verify the required estimates separately for the horizontal and the vertical directions, since we assumed that any two points can be connected by at most five horizontal/vertical segments (the lengths of such segments may also be chosen to be less than the distance of those two points).

First of all, if two points $x_1,x_2$ belong to a common horizontal segment $S$, then we may compute due to the choice of constant speed parametrization that
\[|h_0(x_1) - h_0(x_2)| \leq |x_1 - x_2|\frac{|\mathbf{S}'|}{|S|}.\] 
Note now that due to the fact that $\mathbf{V}'$ had the special structure of being bounded by three concave curves, any shortest curve $\mathbf{S}'$ within $\mathbf{V}'$ must satisfy $|\mathbf{S}'| \leq C \diam(\mathbf{S}')$ for some absolute constant $C$. Moreover, due to the H\"older-estimates for the boundary values $\varphi_0$ we obtain that $\diam(\mathbf{S}') \leq C |s_1 - s_2|^\alpha$. Combining these gives
\[|h_0(x_1) - h_0(x_2)| \leq C |x_1 - x_2|\frac{|s_1-s_2|^\alpha}{|s_1 - s_2|} \leq C |x_1 - x_2|^\alpha.
\]
Let now $x$ and $\tilde{x}$ belong to different horizontal segments $S$ and $\tilde{S}$ but be on the same vertical line. Due to the fact that $H$ only has six sides, we may assume that the left and right endpoints of $S$ and $\tilde{S}$ belong to the same sides of $H$ respectively. We may also assume that
\begin{equation}\label{eq:Sassume}
|\tilde{S}| \leq |S|, \quad |x - \tilde{x}| \leq |S|, 
\end{equation}
the second part being because if this was not the case, we could connect $x$ to $\tilde{x}$ with two horizontal segments along $S$ and $\tilde{S}$ and one boundary curve between $s_1$ and $\tilde{S}$, each of whose length is less than a constant times $|x-\tilde{x}|$. In this case we could conclude the H\"older-continuity estimate simply from the previous estimate for the horizontal direction and the H\"older-continuity of the boundary map.

Now due to the concavity of the boundary curves in $\mathbf{V}'$, the shortest curves $\mathbf{S}' = h_0(S)$ and $\tilde{\mathbf{S}}' = h_0(\tilde{S})$ on the target side in $\mathbf{V}'$ must behave in a specific way. We aim to use this to put the possible shortest curves $\mathbf{S}'$ into a finite amount of different categories in order to be able to assume that $\mathbf{S}'$ and $\tilde{\mathbf{S}}'$ belong to the same category. We do this as follows.

Due to the fact that $\mathbf{V}'$ is bounded by three concave curves, a shortest curve $\mathbf{S}'$ in $\mathbf{V}'$ may only intersect at most one of these concave curves at points different from the endpoints of $\mathbf{S}'$. Thus such a curve $\mathbf{S}'$ consists of up to two line segments starting from either endpoint of $\mathbf{S}'$ and possibly a middle section which goes along one of the boundary curves of $\mathbf{V}'$. We can categorize each such curve by whether it is just a simple segment, a segment combined with a concave part, and so on. We also put curves in a different category depending on which of the three boundary curves of $\mathbf{V}'$ they intersect with. Finally, we may also do the following: If one of the boundary curves of $\mathbf{V}'$ is ''too curved'', meaning the angle of its tangent changes by over $\pi/2$ (in fact, at most one of the three curves may be like this as the total angular change of all three is bounded by $\pi$), then we split it into two halves on which the angle stays below $\pi/2$. We then also categorize shortest curves $S'$ into different categories if they start or end on different halves of such a ''too curved'' boundary part.

As we move the horizontal segment $S$ vertically in the domain, the category of $S'$ can only change a finite amount of times. Note that at the point where a category changes from one to another, there is always one curve $\mathbf{S}'$ which lies at the edge of these two categories and where one or more of its parts (segments or convex part) may be just singletons. We will say that such a curve $\mathbf{S}'$ belongs to both categories to avoid technicalities. In any case, the whole point of putting these curves into categories is that since there are only a finite number of categories and a finite number of ways such categories can appear in order, we may assume that $\mathbf{S}'$ and $\tilde{\mathbf{S}}'$ belong to the same category. Note that the bounds for the finite numbers that appear here are absolute.

\begin{figure}[htbp]
\centering
\includegraphics[width=0.5\textwidth]
{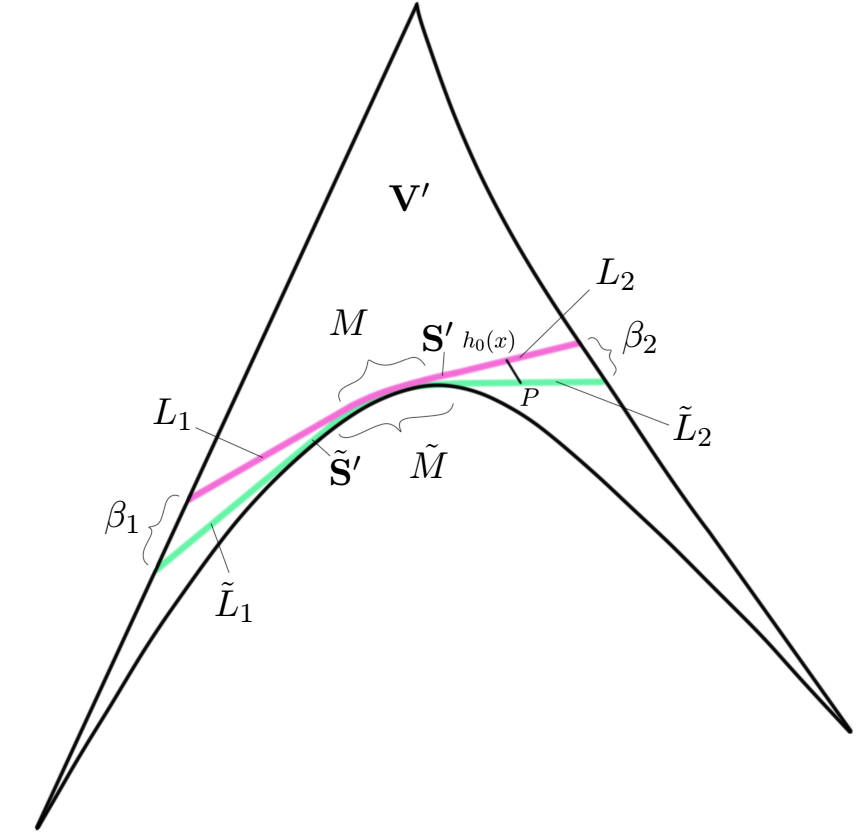}
\caption{Shortest curves $\mathbf{S}'$ and $\tilde{\mathbf{S}}'$ and related objects within the region $\mathbf{V}'$.}
\label{fig:shortest}
\end{figure}

Let $\beta_1$ denote the part of $\partial \mathbf{V}'$ between $\varphi_0(s_1)$ and $\varphi_0(\tilde{s}_1)$ that does not contain $\varphi_0(s_2)$, and define $\beta_2$ analogously, see Fig. \ref{fig:shortest}. Note that due to the fact that $\mathbf{S}'$ and $\tilde{\mathbf{S}}'$ belong to the same category and the fact that we made sure not to include the case where, say, $s_1$ and $\tilde{s}_1$ belong to a ''too curved'' boundary part, we must actually have the length estimate 
\begin{equation}\label{eq:betaineq}
|\beta_1| \leq 2|\varphi_0(s_1) - \varphi_0(\tilde{s}_1)|
\end{equation} and the same for $\beta_2$ as these are convex curves where the angular change is less than $\pi/2$. There are now a few cases to consider depending on which category $\mathbf{S}'$ and $\tilde{\mathbf{S}}'$ belong to, but they are handled very similarly so we only consider the most complicated one.

Let us thus consider the case where $\mathbf{S}'$ and $\tilde{\mathbf{S}}'$ both consist of a line segment, a middle part, and another line segment. We denote this by $\mathbf{S}' = L_1 \cup M \cup L_2$ and $\tilde{\mathbf{S}}' = \tilde{L}_1 \cup \tilde{M} \cup \tilde{L}_2$. By geometry one of the inclusions $M \subset \tilde{M}$ or $\tilde{M} \subset M$ must hold. Since we assumed that $|\tilde{S}| \leq |S|$ in \eqref{eq:Sassume} we should treat these cases separately, but they are similar enough that we will only consider the case $M \subset \tilde{M}$. The situation is now as in Fig. \ref{fig:shortest}.

Given the point $h_0(x) \in \mathbf{S}'$, we let $P$ denote the point on $\tilde{\mathbf{S}}'$ which is obtained from $h_0(x)$ as follows. If $h_0(x) \in L_1$, we traverse from $h_0(x)$ in the direction of the vector $\varphi(\tilde{s}_1) - \varphi(s_1)$ until we hit a point on $\mathbf{S}'$ which is defined as $P$. If $h_0(x) \in M$, then $P = h_0(x)$, and if $h_0(x) \in L_2$, we traverse in the direction of the vector $\varphi(\tilde{s}_2) - \varphi(s_2)$ to hit a point $P \in \mathbf{S}'$. We may suppose that we are in the case $h_0(x) \in L_1$, as the case $h_0(x) \in L_2$ is purely analogous and the case $h_0(x) \in M$ is considered very similarly. Due to the concavity of the curve $\tilde{M}$ it is geometrically clear that we have the internal distance estimate $d_{\mathbf{V}'}(h_0(x),P) \leq |\beta_1|$ and therefore also $|h_0(x) - P| \leq |\beta_1|$. It remains to estimate $|P - h_0(\tilde{x})|$.

Denote by $\partcurve_{\gamma}(a,b)$ the part of a curve $\gamma$ between two points $a,b \in \gamma$. note that since $\tilde{\mathbf{S}}'$ is a shortest curve, we must have the estimate
\[|\partcurve_{\tilde{\mathbf{S}}'}(\varphi_0(\tilde{s}_1),P)| \leq |\beta_1| + |\partcurve_{\mathbf{S}'}(\varphi_0(s_1),h_0(x))| + d_{\mathbf{V}'}(h_0(x),P).\]
Similarly for $\mathbf{S}'$, we get
\[|\partcurve_{\mathbf{S}'}(\varphi_0(s_1),h_0(x))| \leq |\beta_1| + |\partcurve_{\tilde{\mathbf{S}}'}(\varphi_0(\tilde{s}_1),P)| + d_{\mathbf{V}'}(h_0(x),P).\]
Since $d_{\mathbf{V}'}(h_0(x),P) \leq |\beta_1|$, we find the comparison estimate
\begin{equation}\label{comp1}
|\, |\partcurve_{\mathbf{S}'}(\varphi_0(s_1),h_0(x))| - |\partcurve_{\tilde{\mathbf{S}}'}(\varphi_0(\tilde{s}_1),P)| \, | \leq 2 |\beta_1|.
\end{equation}
A similar argument shows that
\begin{equation}\label{comp2}
|\,|\mathbf{S}'| - |\tilde{\mathbf{S}}'|\,| \leq |\beta_1| + |\beta_2|.
\end{equation}
Note next that since $x$ and $\tilde{x}$ are on the same vertical line, since the horizontal segments $S$ and $\tilde{S}$ have endpoints on the same sides of $H$, and since the angle between the sides of $H$ and the horizontal direction is bounded absolutely from below, we have that
\begin{equation}\label{comp3}|\, |x - s_1| - |\tilde{x} - \tilde{s}_1| \, | \leq C_1 |s_1 - \tilde{s}_1| \leq C_2|x - \tilde{x}|.
\end{equation}
By triangle inequality we also have the direct estimate
\begin{equation}\label{comp4}
|\,|S| - |\tilde{S}|\,| \leq |s_1 - \tilde{s_1}| + |s_2 - \tilde{s}_2| \leq C |x - \tilde{x}|.
\end{equation}
Due to the constant speed parametrization of the curves $\mathbf{S}'$ and $\tilde{\mathbf{S}'}$, we may now combine \eqref{comp1}-\eqref{comp4} to calculate that
\begin{align*}
|P - h_0(\tilde{x})| &\leq |\partcurve_{\tilde{\mathbf{S}}'}(P,h_0(\tilde{x}))|
\\ &= \bigl|\,|\partcurve_{\tilde{\mathbf{S}}'}(\varphi_0(\tilde{s}_1),h_0(\tilde{x}))| - |\partcurve_{\tilde{\mathbf{S}}'}(\varphi_0(\tilde{s}_1),P)|\,\bigr|
\\&\leq  \bigl|\,|\partcurve_{\tilde{\mathbf{S}}'}(\varphi_0(\tilde{s}_1),h_0(\tilde{x}))| - |\partcurve_{\mathbf{S}'}(\varphi_0(s_1),h_0(x))|\,\bigr| + 2|\beta_1|
\\&= \left|\frac{|\tilde{s}_1 - \tilde{x}|}{|\tilde{S}|} |\tilde{\mathbf{S}}'|  - \frac{|s_1 - x|}{|S|} |\mathbf{S}'| \right| + 2|\beta_1|
\\&\leq \left|\frac{|\tilde{s}_1 - \tilde{x}|}{|\tilde{S}|}  - \frac{|s_1 - x|}{|S|}\right| |\mathbf{S}'|  + 3|\beta_1| + |\beta_2|
\\&\leq \left|\frac{|\tilde{s}_1 - \tilde{x}| - |s_1 - x|}{|S|}\right| |\mathbf{S}'| + C\frac{|\mathbf{S}'|}{|S|} |x - \tilde{x}|  + 3|\beta_1| + |\beta_2|
\\&\leq C_2 \frac{|\mathbf{S}'|}{|S|} |x - \tilde{x}|  + 3|\beta_1| + |\beta_2|.
\end{align*}
Note now that since the curve $\mathbf{S}'$ is a shortest curve in the domain $\mathbf{V}'$, by the fact that $\mathbf{S}'$ may not turn more than $\pi$ degrees we find the estimate $|\mathbf{S}'| \leq C\diam(\mathbf{S}')$. The H\"older-continuity of the boundary values $\varphi_0$ gives that $\diam(\mathbf{S}') \leq C|S|^\alpha$. Combining these and remembering that we assumed $|x - \tilde{x}| \leq |S|$ in \eqref{eq:Sassume}, for the first term above we have the estimate
\[\frac{|\mathbf{S}'|}{|S|} |x - \tilde{x}| \leq C |S|^{1-\alpha} |x - \tilde{x}| \leq C |x - \tilde{x}|^\alpha.\]
Combining \eqref{eq:betaineq}, \eqref{comp3}, and the H\"older-continuity of $\varphi_0$ gives that
\[|\beta_i| \leq C|x-\tilde{x}|^\alpha,\, i = 1,2.\]
This finally yields that
\[|h_0(x) - h_0(\tilde{x})| \leq |h_0(x) - P| + |P - h_0(\tilde{x})| \leq C|x-\tilde{x}|^\alpha.\]
Hence the extension $h_0$ as defined is H\"older-continuous with the same exponent and comparable constants to the original boundary map $\varphi$.
\vskip 5pt
\textbf{ Step 4. The general construction.} Let us keep the indices $k,n$ fixed for ease of presentation and consider the set $V = V_{n,k}$ defined as the union of $\Gamma = \Gamma_{n,k}$, its two children $\Gamma_+$ and $\Gamma_-$, and the region bounded by them. Note that the curves $\Gamma_+$ and $\Gamma_-$ may overlap with $\Gamma$ and also with each other. In fact, due to the definition of, say, $\Gamma$ and $\Gamma_+$ being shortest curves starting at the point $A$, their intersection is either the single point $\varphi(a)$ or one connected curve $\alpha_+$ starting from $\varphi(a)$. Similarly $\alpha_-$ is defined as the common part of $\Gamma$ and $\Gamma_-$, and $\alpha_{\pm}$ as the common part of $\Gamma_+$ and $\Gamma_-$. The Jordan domain $\hat{V}$ is defined by removing from $V$ the common parts $\alpha_+,\alpha_-,$ and $\alpha_{\pm}$, see Fig. \ref{fig:target}. The boundary $\partial\hat{V}$ consists of three curves which are concave towards its interior due to the assumption that they are non-intersecting shortest curves.
\begin{figure}[htbp]
\centering
\includegraphics[width=0.6\textwidth]
{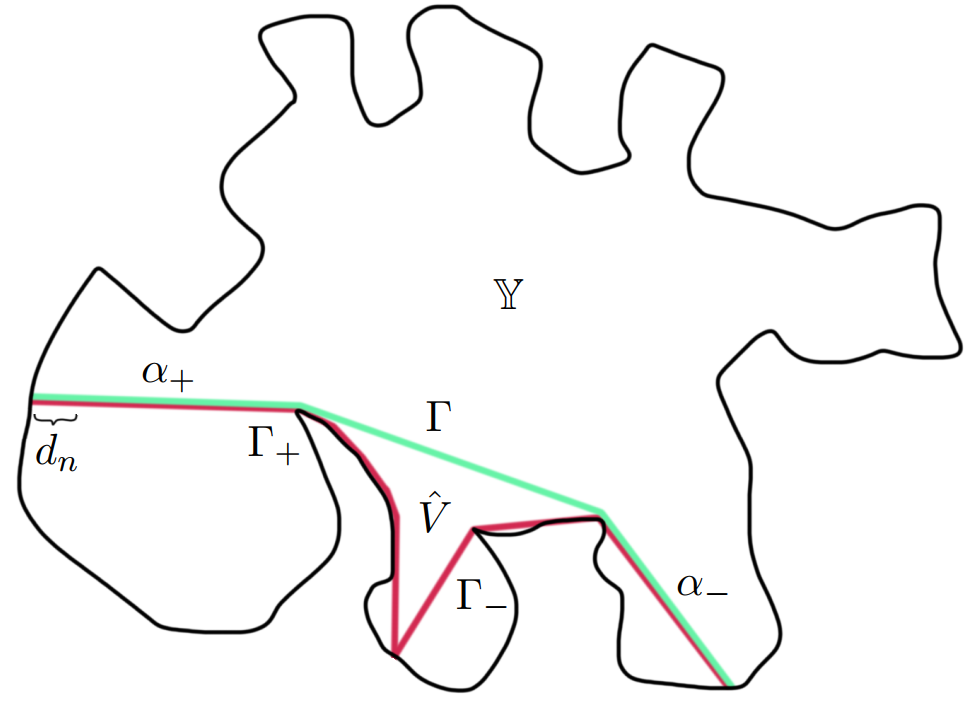}
\caption{The configuration on the target side.}
\label{fig:target}
\end{figure}
Let us now recall the set $U = U_{n,k}$ given on the domain side as the region between the curve $\gamma = \gamma_{n,k}$ (consisting of two sides of the isosceles triangle on top of $I = I_{n,k}$) and its two dyadic children $\gamma_+$ and $\gamma_-$, see Fig. \ref{fig:domain1} for the shape of $U$. In the second step we have defined a parametrization  $\tau : I \to \Gamma$ which is H\"older-continuous with the correct estimates. The map $\tau$ was defined for each dyadic interval and we will sometimes distinguish this map by writing the dependence as $\tau = \tau_{\,\Gamma}$, as opposed to $\tau_{\,\Gamma_+}$ or $\tau_{\,\Gamma_-}$ which are also maps defined on subintervals of $I$ but not necessarily equal to $\tau_{\,\Gamma}$. We would like to lift the parametrization $\tau$ from the interval $I$ to define a parametrization $\tau' : \gamma \to \Gamma$ and repeat this for the children $\gamma_+$ and $\gamma_-$ to define a ''boundary map'' $\tau' : \partial U \to V$, but unfortunately this lifting process cannot be done directly as there may be issues with the fact that the curve $\gamma \subset \partial U$ has very small angles with respect to $\gamma_+$ and $\gamma_-$ at the two endpoints of $I$, causing possible issues with H\"older-estimates if the parametrization $\tau_{\Gamma}$ has some notable difference with $\tau_{\Gamma_+}$ and $\tau_{\Gamma_-}$ near these points.

Instead, we aim to split each curve $\gamma$ into three parts $g_A, g_L$ and $g_R$. The two parts $g_L$ and $g_R$ will be line segments starting from the two endpoints $a,b$ of $I$, and they will be mapped to $\Gamma$ so that their images are very close to the endpoints $\varphi(a)$ and $\varphi(b)$ of $\Gamma$. In fact, at the start of the construction (after fixing the $\Gamma_{n,k}$:s) we fix a sequence $(d_n)$ of numbers decreasing very quickly to zero, and on each dyadic level $n$ the curves $g_L,g_R$ will be mapped to parts of $\Gamma$ which have length $d_n$. The sequence $d_n$ may be chosen to tend to zero arbitrarily fast without impacting H\"older-continuity estimates, and we will use this fact later to be able to say that the extension map $h$ is practically a constant map in the part of $U$ near to the small angles at the endpoints $a,b$ of $I$ and hence we do not need to worry about H\"older-estimates there.

We opt to call the part $g_A \subset \gamma$ the \emph{active part} of $\gamma$, and its image on $\Gamma$ will be most of $\Gamma$. The parts $g_A, g_L$ and $g_R$ are defined as follows, see Fig. \ref{fig:domain1}. If the curve $\gamma$ is not on the first dyadic level, it has a parent curve $\gamma_P$ and shares one endpoint with $\gamma_P$, say the left endpoint $a$ for example. In this case $g_L$ will be the left line segment of $\gamma$, and $g_R$ will be the part of the right line segment of $\gamma$ obtained as a line segment whose vertical projection to the $x$-axis is the same as the right line segment of the rightmost child curve $\gamma_-$. If $\gamma$ is on the first dyadic level, we do the process of defining $g_L$ and $g_R$ via the dyadic children as outlined in the previous sentence for both sides of $\gamma$. The active part $g_A$ is what is left of $\gamma$.
\begin{figure}[htbp]
\centering
\includegraphics[width=1.0\textwidth]
{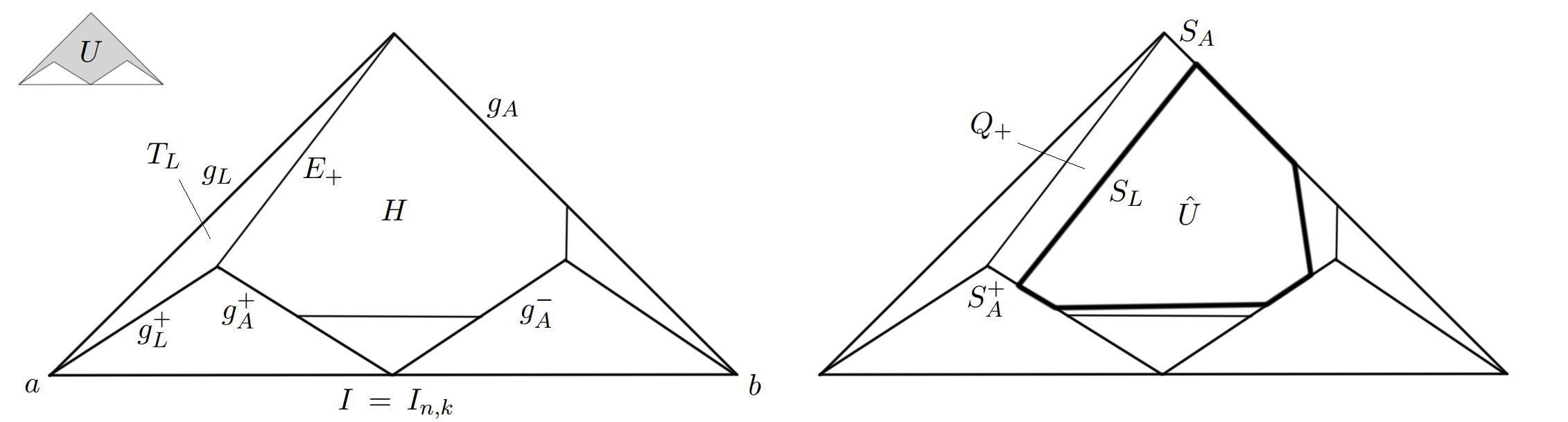}
\caption{Different configurations on the domain side.}
\label{fig:domain1}
\end{figure}
We may connect the active part $g_A$ with the active parts $g^+_A$ and $g^-_A$ of $\gamma_+$ and $\gamma_-$ by taking the convex hull of their union, which defines a hexagonal region $H \subset U$ that is uniformly bilipschitz-equivalent to a ball of radius $2^{-n}$. We now start defining the boundary map $\tau' : \partial U \to \Gamma \cup \Gamma_+ \cup \Gamma_-\subset V$, and we will also define $\tau'$ on the whole of $\partial H$.

Let us again suppose that we are in the situation where $g_L$ is the whole left segment of $\gamma$ and $g_R$ is a part of the right segment of $\gamma$. We first set aside two subcurves $\Gamma_L$ and $\Gamma_R$ of the curve $\Gamma$ which begin at each endpoint and have length $d_n$ and $d_{n+1}$. Note that since both of the endpoints of $\Gamma$ were assumed to be good points in the beginning, there are no technicalities here and, choosing $d_n$ small enough, these parts will always be either line segments or concave curves only touching one side of $\partial \yy$ near their endpoint. Let $\Gamma_m$ denote the remaining middle part of $\Gamma$. We may now make a linear change of variables to transform the parametrization $\tau : \tau^{-1}(\Gamma_m) \to \Gamma_m$ into a parametrization $\tau' : g_A \to \Gamma_m$ from the active part of $\gamma$ to $\Gamma_m$ (in the case where we are on the first dyadic level, a minor modification to this is sufficient). The H\"older-estimates are inherited since $|g_A| \approx |I|$.

We next turn to define $\tau'$ on the curves $g_L$ and $g_R$ which will be mapped to $\Gamma_L$ and $\Gamma_R$ respectively. We would like to impose the following rule which affects parametrizations across all dyadic levels: If a point $Q \in g_L$ and a point on $Q_+ \in g^+_L$ (with $g^+_L$ denoting the part $g_L$ for the child curve $\gamma_+$) have the same projection to the $x$-axis, then they should be mapped to points on $\tau'(Q) \in \Gamma$ and $\tau'(Q_+) \in \Gamma_+$ which are "equally far away from $\varphi(a)$" meaning that
\begin{equation}\label{eq:projX}|\partcurve_{\Gamma}(\varphi(a),\tau'(Q))| = |\partcurve_{\Gamma_+}(\varphi(a),\tau'(Q_+))|.\end{equation}
We also impose an analogous condition for pairs $Q \in g_R$, $Q_- \in g^-_R$ having the same projection to the $x$-axis.

Note that $\Gamma$ and $\Gamma_+$ may be the same curve near their mutual endpoint $\varphi(a)$, in which case this just says that the images of the points $Q,Q_+$ under $\tau'$ are equal. In fact, we make the following standing assumption here. If $\Gamma$ and $\Gamma_+$ do overlap, meaning $\alpha_+$ is not equal to a singleton, then we choose $d_n$ so small that $d_n \leq |\alpha_+|$ (here $\alpha_+$ depends on $n$ but also $k$, so we take a minimum over $k$). We also assume $d_n \leq |\alpha_-|$ and $d_n \leq |\alpha_\pm|$ as long as the curves are not singletons. This forces us to always map $g_R$ and $g_L$ under $\tau'$ to subsets of $\alpha_+$ or $\alpha_-$ if these aren't singletons.

The consequence of such a rule is that to define $\tau'$ on $g_L$ and $g_R$ it is enough to define $\tau'$ on the part of $g_L$ that does not project to any point that is also a projection from the child curve $g^+_L$. The definition of the map $\tau'$ on any other point of $g_L$ will be obtained inductively from successive dyadic generations, and since we were in the case where $g_R$ projects to the same set as $g^-_R$ we need not define $\tau'$ on any point on $g_R$ yet. On the part of $g_L$ where $\tau'$ needs to be defined, we may simply set it to map this part (which is a segment) with constant speed to the part of $\Gamma$ which contains the points of length $d_{n+1}$ to $d_n$ from the endpoint $\varphi(a)$. In practice this means that on $g_L$ the map $\tau'$ will consist of maps defined on a sequence of segments (see Fig. \ref{fig:parts}) tending to the endpoint $a$ in dyadic fashion having length comparable to $2^{-N}$, $N = n, n+1, \ldots$, and each such segment is mapped with constant speed to a part of $\Gamma$ with length $d_N - d_{N+1}$.

\begin{figure}[htbp]
\centering
\includegraphics[width=0.6\textwidth]
{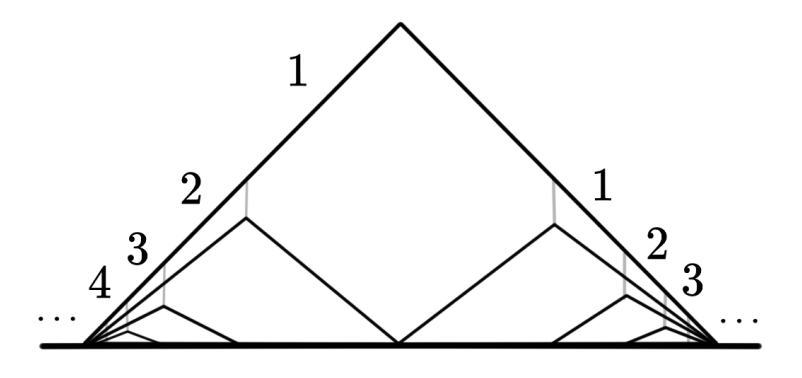}
\caption{Dividing the segments $g_L$ and $g_R$ into dyadic parts on which to define $\tau'$.}
\label{fig:parts}
\end{figure}

Hence $\tau'$ has been defined on all parts of $\gamma$ (and the curves $\gamma$ for successive dyadic generations). We next define $\tau'$ on $\partial H$, where we only need to define $\tau'$ on the three sides of the hexagon $\partial H$ not already contained in $\gamma, \gamma_+,$ or $\gamma_-$. Let us call these segments $E_+$, $E_-$, and $E_\pm$, and the definition of $\tau'$ on each such segment is similar enough that we may only do so on $E_+$ which is the segment between the non-mutual endpoints $a_L \in g_L$ and $a^+_L \in g^+_L$, see Fig. \ref{fig:domain1}. To define $\tau'$ on $E_+$ we will define an image curve $F_+$ within $V$ so that $\tau' : E_+ \to F_+$ is simply a constant speed parametrization. The goal is simply that the length of $F_+$ is at most $2 d_n$. But since the internal distance (meaning path distance within the closed set $V$ here) of the points $\tau'(a_L)$ and $\tau'(a^+_L)$ within $V$ has to be at most $d_n + d_{n+1}$ since
\[|\partcurve_{\Gamma}(\varphi(a),\tau'(a_L))| = d_n \quad \text{ and } \quad |\partcurve_{\Gamma_+}(\varphi(a),\tau'(a^+_L))| = d_{n+1},\]
this choice of $F_+$ is always possible. More precisely we define $F_+$ to be the shortest curve in $\overline{V}$ between its endpoints $\tau'(a_L)$ and $\tau'(a_{L}^+)$.

Note that if $\Gamma$ and $\Gamma_+$ overlap (so that $\alpha_+$ is not a singleton) then $F_+$ may simply be chosen as the curve $\partcurve_{\alpha_+}(\tau'(a_L),\tau'(a^+_L))$. Here $\tau'(a_L),\tau'(a^+_L) \in \alpha_+$ by our earlier assumption that if $\alpha_+$ is not a singleton then $d_n < |\alpha_+|$.

Thus now $\tau'$ has been defined as a constant speed map from $E_+$ to $F_+$ and since $d_n$ was chosen small enough we can assume to have the same H\"older-estimates for $\tau'$ on $E_+$ as the original boundary map.

It is now simple to also verify the H\"older-continuity of $\tau'$ on the whole of $\partial U \cup \partial H$. On each of the six sides of $\partial H$ the map $\tau'$ was already verified to have the correct H\"older-estimates, and since the hexagon $\partial H$ has no small angles this holds on the whole boundary as well. All the other parts of $\partial U$ are mapped to curves where the small enough choice of the sequence $(d_n)$ will guarantee H\"older-continuity.

Next, we finally define the extension $h$ by extending the map $\tau'$ to the whole of $U$. Let us start with the region $H \subset U$. We must first account for the possible overlapped parts $\alpha_+,\alpha_-,$ and $\alpha_\pm$ in the target. Let us take the whole preimage $B$ of these parts under $\tau'$. Let us for now make a standing assumption that none of these parts are singletons, and we will cover the other case later. We explain here how this affects things around the preimage of $\alpha_+$ and the cases of $\alpha_-$ and $\alpha_\pm$ are handled very similarly. Since $g_L$ and $g^+_L$ are part of this preimage $B$, there are also some points on the boundary of $H$ which are mapped to $\alpha_+$: the curve $E_+$ in particular belongs to $B$ but also some part of $g_A$ and $g^+_A$ belongs to $B$. We may, however, assume that $g_A$ is not fully contained in the preimage $B$ - this may be achieved simply by choosing $d_n$ small enough so that $|\Gamma| - |\alpha_+| > d_n$. Similarly we may assume that $g^+_A$ and $g^-_A$ are not contained in $B$. Hence there are some subsets (line segments) of $g_A, g^+_A,$ and $g^-_A$ which are not in $B$, and these three line segments themselves bound yet another hexagon which we denote by $\hat{U}$, see Fig. \ref{fig:domain1}. Now $\hat{U}$ is still convex but may be fairly thin for example.

On three of the sides of $\hat{U}$ (the sides on the active parts) the map $\tau'$ was already defined. For each of the other three sides, the map $\tau'$ has the same value on both endpoints (this value is the endpoint of $\alpha_+, \alpha_-,$ or $\alpha_\pm$ not on $\partial \yy$) so we extend $\tau'$ as a constant map on these three sides. Now $\tau'$ has been defined on the whole boundary of the convex hexagon $\hat{U}$, and the image of $\partial \hat{U}$ under $\tau'$ is the boundary of the region $\hat{V}$ bounded by three concave curves. Hence at this part of the construction we may repeat the idea of the initial case from Step 3 where exactly the same situation was addressed to extend a boundary map from a convex hexagonal region to a region bounded by three concave curves. The H\"older-estimates obtained there show that $h$ is H\"older-continuous in $\hat{U}$ with the correct estimates.

Now the set $H \setminus \hat{U}$ divides into up to three convex quadrilateral regions. For simplicity we consider here the case of such a region $Q_+$ which has one side given by $E_+$ and shares points with the preimage of $\alpha_+$. In fact, the whole boundary of $Q_+$ is mapped to $\alpha_+$ under $\tau'$. Let us label the sides of $Q_+$ as $S_L, S_A, S^+_A,$ and $E_+$. Here $S_A \subset g_A$ and $S^+_A \subset g^+_A$. We wish to first compare how the map $\tau'$ behaves on the sets $S_A$ and $S^+_A$, both mapped to $\alpha_+$ injectively. On the respective endpoints of $S_A$ and $S^+_A$ that also belong to $S_L$, $\tau'$ attains the same value. The other endpoints of $S_A$ and $S^+_A$ 
(which are $a_L$ and $a^+_L$) 
are mapped to the curve $\alpha_+$ and are of curve length $d_n$ and $d_{n+1}$ away from the endpoint $\varphi(a)$ respectively. We distinguish a subsegment $Z \subset S^+_A$ for which  $\tau'(g_L^+)\cup\tau'(Z) = \tau'(g_L)$, meaning that 
the endpoint of $Z$ not on $E_+$ is also mapped to curve length $d_n$ away from $\varphi(a)$. We then need a helpful remark.

\textbf{Remark:} Given a point $p$ on $\varphi(I)$ and letting $q$ denote the point on $\Gamma$ which is the closest point to $p$ w.r.t. the internal distance in $\overline{\yy}$, we have the following: If $q$ lies on $\alpha_+$ but is not an endpoint of $\alpha_+$, then $q$ is also the closest point to $p$ on $\Gamma_+$ w.r.t. the internal distance, see Fig. \ref{fig:shorter1} on the left. To see this, consider the minimal boundary component of $\partial \yy$ which contains $p$ and has endpoints on $\Gamma_+$. Then the part of $\Gamma_+$ between these two endpoints has to be concave towards $p$ and the closest point $q'$ on $\Gamma_+$ to $p$ w.r.t the internal distance lies on this part. Now if $q' \neq q$, the three curves given by $\Gamma_+$ and the two shortest curves $\Gamma_q$ from $p$ to $q$ and $\Gamma_{q'}$ from $p$ to $q'$ bound a region $\Delta$ whose three boundary curves are all concave towards the interior. Moreover, since $q'$ was the closest point to $p$ on $\Gamma_+$, it must hold that the angle between $\Gamma_{q'}$ and $\Gamma_+$ towards the interior of $\Delta$ is at least $\pi/2$. Since $q$ was not an endpoint of $\alpha_+$ and $\alpha_+ \subset \Gamma_+$, it must also hold that the angle between $\Gamma_q$ and $\Gamma_+$ is at least $\pi/2$. Now the sum of two interior angles in the region $\Delta$ is at least $\pi$, which gives a contradiction as this cannot happen when each of the boundary curves of $\Delta$ is concave towards the interior.

Note, however, that if $q$ is the other endpoint of $\alpha_+$, it is possible that $q' \neq q$ if the situation is as in Fig. \ref{fig:shorter1} on the right. This is the only exceptional case, and to avoid additional consideration of this case in the future, we will simply choose to make an arbitrarily small modification to the parametrization map $\tau$ ($= \tau_{\,\Gamma}$) in this situation which does not affect the H\"older-estimates for $\tau$ (and by proxy $\tau'$). This modification is done in a way where the preimage of all such exceptional points $p$ under $\varphi$ is not mapped to $q$ by $\tau_{\,\Gamma}$ but rather to an arbitrarily small part of the curve $\Gamma$ starting at the end of $\alpha_+$.
\begin{figure}[htbp]
\centering
\includegraphics[width=1.0\textwidth]
{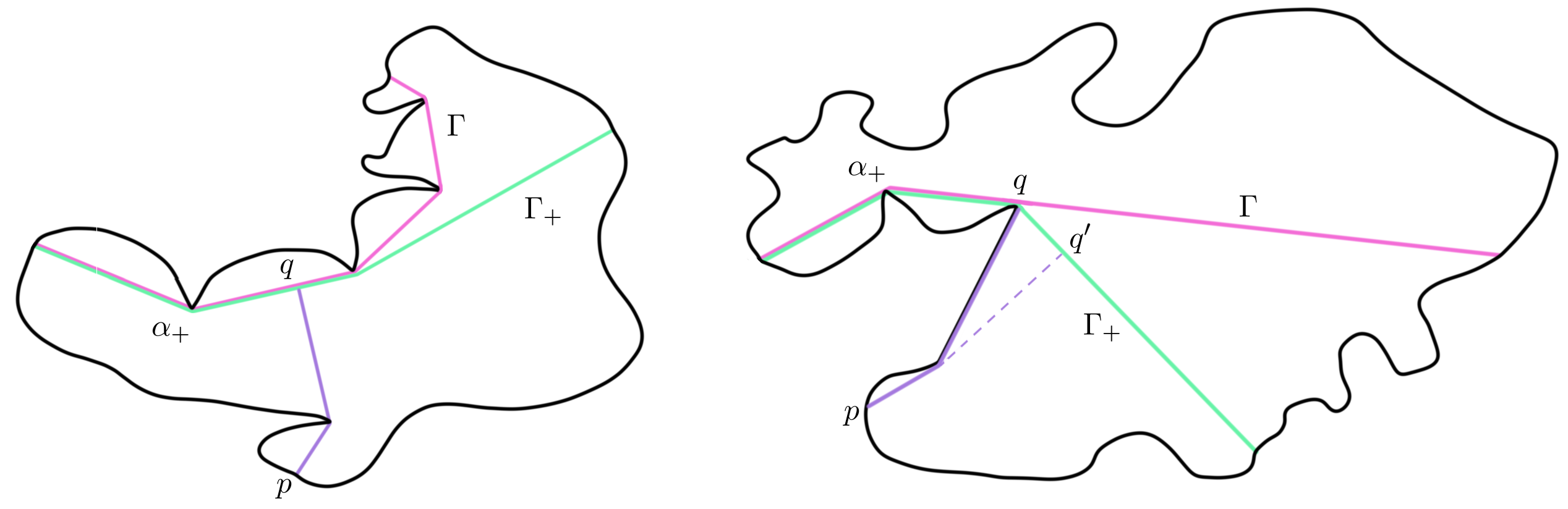}
\caption{The shortest paths from $p$ to $q \in \alpha_+$. If $q$ is an endpoint of $\alpha_+$, the situation may change.}
\label{fig:shorter1}
\end{figure}

The conclusion we make from the remark in the previous two paragraphs is that the parametrization map $\tau$ which can be defined for both the curve $\Gamma$ and the curve $\Gamma_+$ (we distinguish these maps by $\tau$ and $\tau_+$) behaves exactly the same on the preimage of $\alpha_+$, meaning that $\tau(x) = \tau_+(x)$ for $x \in \tau^{-1}(\alpha_+)$. Hence the map $\tau'$ which was obtained from the map $\tau$ by a linear change of variables also satisfies that $\tau'\vert_{S_A}$ is simply a linear change of variables of $\tau'\vert_{Z}$. 
This lets us divide $Q_+$ into a triangle $T'$ with two sides given by $Z$ and $E^+$ and a quadrilateral $Q'$ with three sides $S_A$, $S_L$ and $S^+_A \setminus Z$.

On the quadrilateral $Q'$, we simply define $h$ as a linear interpolation between the maps $\tau'\vert_{S_A}$ and $\tau'\vert_{S^+_A \setminus Z}$. More precisely, since $Q'$ may be represented as $Q' = \{t S_A + (1-t) (S^+_A \setminus Z) : 0 \leq t \leq 1\}$, we define $h$ on each intermediate segment $t S_A + (1-t) (S^+_A \setminus Z)$ as simply the map $\tau'$ composed with the natural projection from $t S_A + (1-t) (S^+_A \setminus Z)$ to $S_A$. This gives a H\"older-continuous map $h : Q' \to \alpha_+$ with estimates inherited from $\tau'$ since the map is practically just $\tau'$ in one direction and constant in the other.

On the triangle $T'$, the map $\tau'$ maps the two sides $E^+$ and $Z$ both onto the same part of $\alpha_+$ between $d_{n+1}$ and $d_n$ length from $\varphi(a)$. The parametrizations here may be slightly different as the map $\tau'$ on $E^+$ was simply defined via constant speed parametrization while $Z$ it is inherited from the map $\tau$. Nevertheless, since both parametrizations satisfy appropriate H\"older-estimates and due to $d_n$ being small enough, we may simply linearly interpolate between these two parametrizations in $T'$ to define an extension map $h: T' \to \tau'(Z) \subset \alpha_+$ with the correct H\"older-estimates.



This defines the map $h$ on $Q_+$ and hence it is now defined on the whole hexagon $H \subset U$. The same idea as in the previous paragraph may be applied to define the extension $h$ on the triangle $T_L$ with sides $g_L$ and $g^+_L$. In fact, if we remove from $g_L$ the part that projects to the $x$-axis as the same set as $g^+_L$, this splits the triangle $T_L$ into two triangles. In one triangle we may use the $x$-axis projection property \eqref{eq:projX} to define the extension $h$ simply as a constant map on vertical segments, while on the remaining triangle (whose one side is $S_L$) we may use the same idea as in the previous paragraph. This defines the extension $h$ on $T_L$ and the same idea may be used to define $h$ on the whole of $U$.

Let us now address the case where $\alpha_+$ is a singleton. Then we would modify the above construction as follows. First, we note that the curves $\Gamma$ and $\Gamma_+$ must be concave towards each other near their common point $\varphi(a)$, so we may choose the earlier curve $F_+$ as simply a line segment between the respective endpoints given by points at length $d_n$ and $d_{n+1}$ on $\Gamma$ and $\Gamma_+$ respectively (if needed, we can decrease $d_n$ beforehand so that this segment $F_+$ has no additional intersections with other curves). We recall that $\tau'$ maps the edge $E_+$ of the hexagon $H$ to this segment $F_+$, and in this case $E_+$ is also an edge of the hexagon $\hat{U}$ mapped to $\hat{V}$ previously. But in this case we must instead construct the extension $h$ to map $\hat{U}$ to the set $W$ obtained by removing from $\hat{V}$ the small part $W_+$ near $\varphi(a)$ that is cut away from it by the segment $F_+$. If some of the curves $\alpha_-$ or $\alpha_\pm$ are also singletons, this is reflected in the definition of $W$ by also cutting away appropriate parts $W_-$ and/or $W_\pm$.

Hence the set $W$ consists three concave boundary curves and up to three other boundary parts which are arbitrarily small segments between these curves. We would like to still use the idea from the 'initial case' to define an extension $h : \hat{U} \to W$, but the argument is not applicable without some modifications as it assumed that the target boundary did not have the additional small segments such as $F_+$. Nevertheless, one may verify that as long as these segments are small enough there is no essential change to the argument or estimates. Hence this is just a matter of choosing the numbers $d_n$ to be appropriately small depending on the curve $\Gamma$ and its children. In fact, one may even repeat the arguments of the initial construction in this situation and notice that the shortest curves in $W$ are still only consisting of at most two segments and a concave part in the middle, so the proof of H\"older-estimates will go through similarly as well.

On the triangle $T_L$, we define the map $h: T_L \to W_+$ as an extension of the boundary values $\tau'$. For such an extension we may actually use the idea from the 'initial case' since $W_+$ is bounded by two concave curves and a segment. But any reasonable extension method here will also work as long as the number $d_n$ is chosen small enough since any bounds obtained will get uniformly smaller as $d_n$ tends to zero. Some care should also be taken here to respect the way $\tau'$ was defined dyadically on the boundary curves $g_L$ and $g_L^+$ of $T_L$, recall Fig. \ref{fig:parts}. Each such piece of $g_L$ defines a "dyadic slice" of $T_L$ that projects to the same region on the $x$-axis as $g_L$, and we may define the extension $h$ on $T_L$ so that on the $m$:th dyadic slice we have estimates controlled by $d_{n+m-1}$ instead of $d_n$. This observation is fairly trivial but necessary to get the correct global bounds as we approach the boundary.
\vskip 5pt
\textbf{Step 5. Global H\"older continuity.} In the previous part of the construction, we were able to define the extension $h$ in each set $U = U_{n,k}$ and verified that it satisfies the appropriate H\"older-continuity estimates locally in each such set. The fact that $h$ maps $U_{n,k}$ as a monotone map to $V_{n,k}$ means that this already defines $h$ as a monotone map of the whole disk $\dd$ to the target domain $\yy$. In this part of the proof we will verify that the map $h$ is also H\"older-continuous in the whole disk $\dd$ with the appropriate estimates inherited from its local behaviour.

To do this, consider two points $x,y \in \dd$. Thus $x \in U_{n,k}$ and $y \in U_{n',k'}$ for some $n,k,n',k'$, where we may suppose that $n \leq n'$. Let $N$ be such that $2^{-N} \approx |x-y|$. We split into two cases depending on how far $x$ and $y$ are apart compared to the level of their dyadic sets, see Fig. \ref{fig:casesxy}.

\emph{Case 1.} $N \leq n$. In this case, we may traverse from $x$ to $y$ via two sequences of sets $U^n_x, U^{n-1}_x, \ldots, U^{N}_x$ and $U_y^{N}, U^{N+1}_y, \ldots, U^{n'}_y$. Each of these sets is one of the dyadic sets $U_{m,l}$ considered before, with $x \in U^n_x, y \in U^{n'}_y$ and the middle set $U^N_x = U^N_y$ is from the dyadic level $N$ and hence having diameter comparable to $2^{-N}$. Moreover, the dyadic level changes by at most one in each step which allows us to use the H\"older-continuity of $h$ in each set to estimate that
\begin{align*}
|h(x) - h(y)| &\leq \sum_{l=n}^{N}\diam(h(U^l_x)) + \sum_{l=n'}^{N} \diam(h(U^l_y))
\\&\leq C\sum_{l=n}^{N}2^{-l\alpha} + C\sum_{l=n'}^{N} 2^{-l\alpha}
\\&\leq 2C\sum_{l=N}^\infty 2^{-l\alpha}
\\&\leq C' 2^{-N\alpha}.
\end{align*}
Since $2^{-N} \approx |x-y|$, this gives the correct H\"older-continuity estimate.
\begin{figure}[htbp]
\centering
\includegraphics[width=0.85\textwidth]
{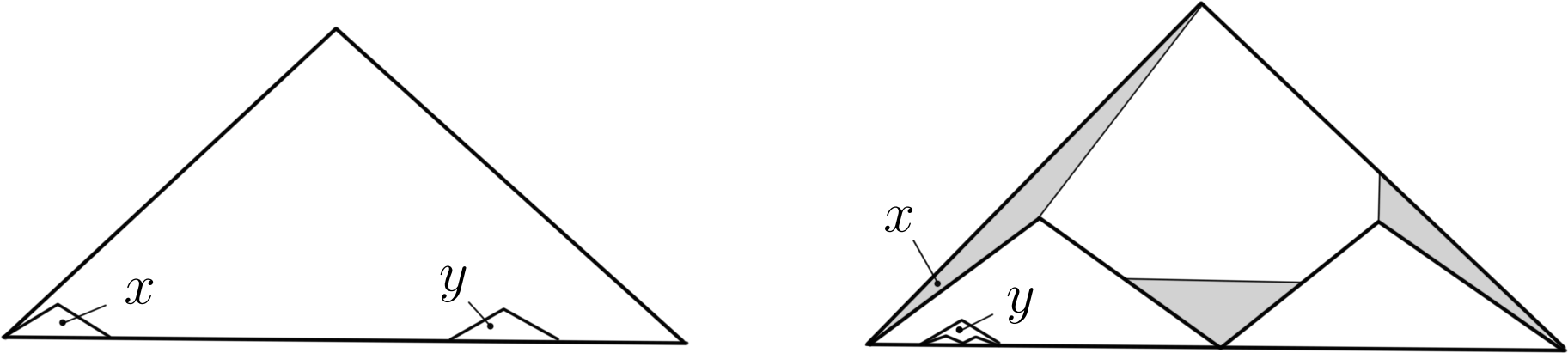}
\caption{The two cases for positioning of $x$ and $y$. On the right, the triangular regions mapped very close to the boundary are highlighted.}
\label{fig:casesxy}
\end{figure}

\emph{Case 2.} $N > n$. We consider a few possibilities. If $x$ and $y$ are in the same set $U_{n,k}$ or $y$ is in one of the two child sets of $U_{n,k}$, then there is nothing to prove as we only need to use the local H\"older estimate for $h$ at most twice to get the required estimate.

Otherwise $x$ must be in one of the triangular regions in $U_{n,k}$ which were mapped to sets of distance at most $d_n$ from the corresponding boundary point under $h$, see the right part of Fig. \ref{fig:casesxy}, as otherwise the distance $|x-y|$ would be too large.

If now $U_{n',k'}$ does not share a boundary point with $U_{n,k}$ (as in the right part of Fig. \ref{fig:casesxy}), then we can argue using the idea from Case 1 to assume that it does. Indeed, Case 1 lets us replace $y$ with any point $y'$ belonging to a set $U_{n',m}$ within the same generation $n'$ as long as $|y-y'| \approx 2^{-N}$. Geometrically it is simple to see that if $U_{n',k'}$ didn't already share a boundary point with $U_{n,k}$ then we may pick the new point $y'$ from such a set instead. Hence we are assumed to be in a configuration similar to Fig. \ref{fig:casesxy2}.

%

\begin{figure}[htbp]
\centering
\includegraphics[width=0.85\textwidth]
{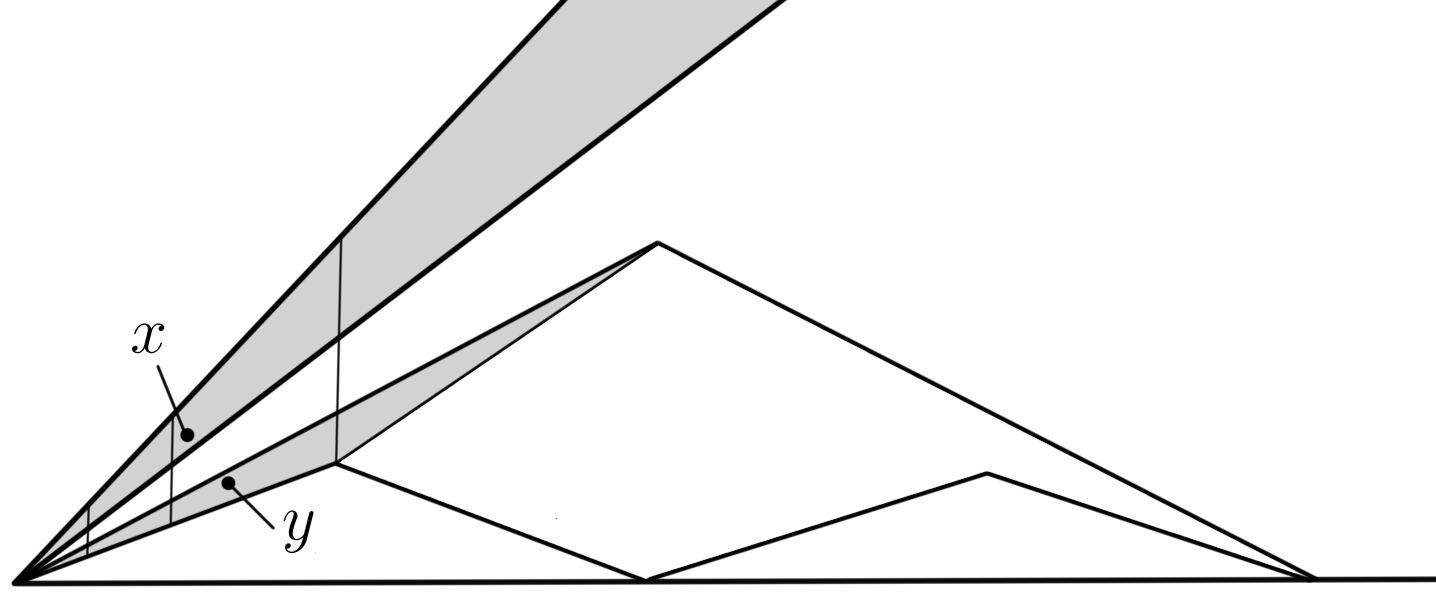}
\caption{Points $x$ and $y$ very close to each other.}
\label{fig:casesxy2}
\end{figure}

In this case, we must recall the construction of the map $h$ inside such triangular regions, where it was defined dyadically by splitting the region into dyadic slices depending on their projection to the $x$-axis as in Fig. \ref{fig:parts} and the discussion at the end of Step 4. Let us suppose that $x$ lies in the $m$:th dyadic slice, mapped to a small set of diameter less than $d_{n+m-1}$. We may also suppose that $y$ belongs to the corresponding dyadic slice within $U_{n',k'}$, as otherwise we may again replace $y$ by such a point while this time using the H\"older-estimates inside $U_{n',k'}$.

Now both $x$ and $y$ are mapped within distance $d_{n+m-1}$ of each other. We may now conclude by assuming that the sequence $d_n$ was picked to be small enough so that we get the desired estimate. To be precise, since we assumed that $x$ and $y$ belonged to corresponding dyadic slices which were not neighbours, the distance $|x-y|$ is bounded from below in terms of some function $n+m-1$ (depending on our initial choice of angles for the curves $\gamma_{n,k}$ on the domain side). Hence the inequality $d_{n+m-1} \leq C|x-y|^\alpha$ may be obtained by small enough choice of sequence $(d_n)$.
\vskip 5pt
\textbf{Step 6. Injectification of the map.} In this part of the proof we explain how the map $h$ may be modified slightly to produce a homeomorphic extension $H : \dd \to \yy$ instead. The map $h$ itself has the potential issue of not being injective on each set $U_{n,k}$. For example, in the case where $\Gamma = \Gamma_{n,k}$ and its child curve $\Gamma_+$ intersect along a curve $\alpha_+$ then the preimage of $\alpha_+$ will be a set where we have loss of injectivity. The curves $\Gamma_{n,k}$ may also touch $\partial \yy$ outside of their endpoints as they were simply defined as shortest curves inside $\overline{\yy}$.

But since the endpoints of each curve $\Gamma_{n,k}$ were 'good points' in $\partial\yy$, it is fairly clear that we may move each curve slightly towards the interior of $\yy$ and even replace each such curve by a piecewise linear approximation instead. Here we may choose to make the piecewise linear approximation such that it produces a finite number of pieces in each compact subset of $\yy$, possibly allowing for fine adjustments as we near the boundary at the endpoints of $\Gamma_{n,k}$ (to account for the dyadic way the map $h$ was defined near the preimages of such endpoints, recall again Fig. \ref{fig:parts}). This may also be done in such a way that none of the curves $\Gamma_{n,k}$ intersect within the interior of $\yy$ and hence produce topologically the same configuration as the curves $\gamma_{n,k}$ on the domain side. It is also clear that such modifications may be done in an arbitrarily small fashion on each dyadic level, if necessary so that the adjustment also gets smaller and smaller as we near an endpoint of $\Gamma_{n,k}$.

The extension map $h$ will be modified a follows. First, we adjust the parametrization $\tau'$ to map instead to the new image curves $\Gamma_{n,k}^*$ in place of $\Gamma_{n,k}$. At this point we may also injectify the parametrization $\tau'$ in parts where it may have not been injective before (certain intervals could be mapped into single points, recall Fig. \ref{define}). This can also be done in an arbitrarily small way, again adjusting for increasingly smaller changes as we near the boundary.
\begin{figure}[htbp]
\centering
\includegraphics[width=0.85\textwidth]
{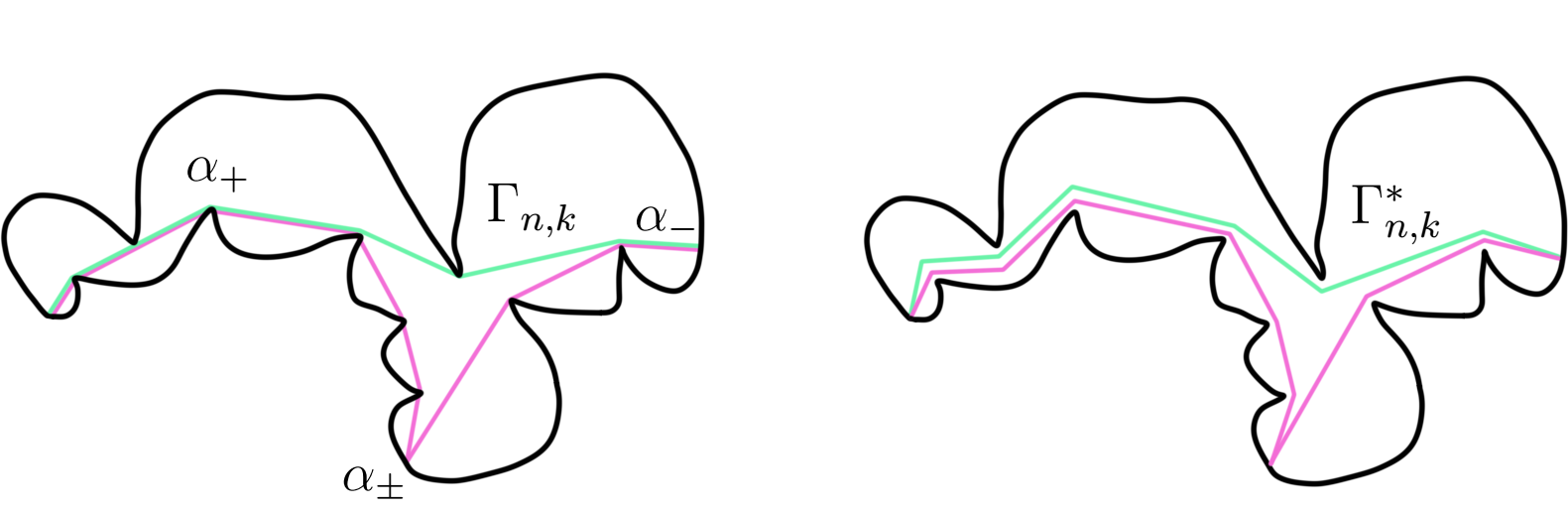}
\caption{Injectification of the map.}
\label{fig:injectify}
\end{figure}
On the sets where $h$ was already injective before, it is fairly clear that the aforementioned changes to the curves $\Gamma_{n,k}$ only produce very small perturbations of the construction of $h$. On the preimages of the possible overlapped curves $\alpha_+, \alpha_-,$ and $\alpha_\pm$ under $h$, we must redefine $h$ as a map $H$ which now sends the preimages to an appropriate region defined by the two replacements of such curves. For example, if originally $\alpha_+$ was the intersection of $\Gamma$ and $\Gamma_+$, then after the fact that we have replaced $\Gamma$ and $\Gamma_+$ by two curves $\Gamma^*$ and $\Gamma_+^*$ that do not intersect each other, there are parts $\alpha^* \subset \Gamma^*$ and $\alpha^*_+ \subset \Gamma^*_+$ which correspond to the part that was $\alpha_+$ before. Since the adjustment may be done in an arbitrarily small way, the other endpoints of the curves $\alpha^*$ and $\alpha^*_+$ (which are now distinct points) may be connected by an arbitrarily small curve $\ell^*$, in fact a line segment if the earlier perturbation is done appropriately. Furthermore within the region defined by $\alpha^*, \alpha^*_+$ and $\ell^*$ it will be possible to define an isotopy $\alpha^t$, $t \in [0,1]$, between $\alpha^*$ and $\alpha^*_+$ with the other endpoint of $\alpha^t$ moving on the segment $\ell^*$.

Recalling that the map $h$ was defined on all preimages of the curve $\alpha_+$ as some form of homotopy between different parametrizations of $\alpha_+$. For example, in the quadrilateral $Q' \subset Q_+$ the map $h$ was simply defined by linear interpolation between the boundary map $\tau'$ on two opposite sides of $Q'$. If we now modify this linear interpolation by mapping to the curves $\alpha^t$, where $t$ runs over the whole interval $[0,1]$, instead then this produces an injective version of the map $h$ in this set. As usual, since the modifications were arbitrarily small, the change to the H\"older-constant may also be assumed to be arbitrarily small here.
\end{proof}

\end{document}